\documentclass[letterpaper, 10 pt, conference]{ieeeconf}  

\IEEEoverridecommandlockouts                              

\makeatletter
\let\NAT@parse\undefined
\makeatother

\usepackage{mathptmx}       
\usepackage{makeidx}         
\usepackage{graphicx}        
\usepackage{multicol}        
\usepackage{amsmath,amssymb,wrapfig}
\usepackage{euscript}
\usepackage{mathrsfs}
\usepackage{array}
\usepackage{gatechamberlab}
\usepackage{epstopdf}
\usepackage{multirow}
\usepackage{subfig}
\usepackage{arydshln}
\usepackage{dblfloatfix}    
\usepackage[colorlinks]{hyperref}
\hypersetup{
    colorlinks,
    linkcolor=blue,
    citecolor=blue,
    filecolor=magenta,
    urlcolor=cyan,
}
\usepackage{accents}
\newcommand{\ubar}[1]{\underaccent{\bar}{#1}}

\title{\LARGE \bf
Input to State Stability of Bipedal Walking Robots: Application to DURUS
}

\author{Shishir Kolathaya, Jacob Reher and Aaron D. Ames
\thanks{This work is supported by the National Science Foundation through grants NRI-1526519}
\thanks{Shishir Kolathaya, Jacob Reher are with the School of Mechanical and Civil Engineering, California Institute of Technology,
        Pasadena, CA, USA
        {\tt\small \{sny,jreher\}@caltech.edu}}%
\thanks{Aaron D. Ames is with the Faculty of the School of Mechanical and Civil Engineering, California Institute of Technology,
        Pasadena, CA, USA
        {\tt\small ames@caltech.edu}}%
}

\begin{document}

\maketitle
\thispagestyle{empty}
\pagestyle{empty}

\begin{abstract}

Bipedal robots are a prime example of systems which exhibit highly nonlinear dynamics, underactuation, and undergo complex dissipative impacts. 
%
%
This paper discusses methods used to overcome  a wide variety of uncertainties, with the end result being stable bipedal walking.
%
%
%
%
%
The principal contribution of this paper is to establish sufficiency conditions for yielding input to state stable (ISS) hybrid periodic orbits, i.e., stable walking gaits under model-based and phase-based uncertainties.
In particular, it will be shown formally that exponential input to state stabilization (e-ISS) of the continuous dynamics, and hybrid invariance conditions are enough to realize stable walking in the $23$-DOF bipedal robot DURUS.
%
This main result will be supported through successful and sustained walking of the bipedal robot DURUS in a laboratory environment.




\end{abstract}

\section{Introduction}
\label{sec:introduction}

Bipedal locomotion techniques such as 
zero moment point (ZMP) \cite{VB05}, capture point \cite{6094435} and linear inverted pendulums \cite{5651082}, rely on a restricted set of motions to simplify the robot dynamics, 
i.e., forcing the center of mass (COM) to stay at a constant height. 
Enforcing constant COM height renders linear dynamics on the reduced order model. 
Subsequently, with this constraint, linear controllers can be applied by satisfying only one basic criterion: bounded-input-bounded-output stability (BIBO). While these methods help to increase robustness to the highly uncertain nonlinear dynamics, the resulting locomotive behaviors are often quasi-static and slow. 
Aiming to address these constraints, the bipedal walking community has worked towards 
utilizing the fullbody dynamics of the system 
in order to achieve complex behaviors that are not only fast but also very efficient. 
Several examples of successful realizations include 
\cite{umich_mabel,Hubicki2016,reheralgorithmic}. 

In order to realize dynamic behaviors such as running and dancing, it becomes necessary to exploit the natural nonlinear dynamics of the robot. With this goal, reduced order models and heuristics have been successfully used to design dynamic behaviors \cite{Hubicki2016,ACMRAIBERT}. On the other side, methods utilizing more formal methods have limited practical results and include control Lyapunov functions (CLFs) \cite{TAC:amesCLF}, combinations of control Lyapunov functions and control barrier functions (CBF) \cite{RSS2015_RobustCLF}. 
With a view toward exploring a more formal approach, we will identify and realize robust walking controllers 
%
that satisfy the equivalent of BIBO stability criterion for nonlinear systems; the input to state stability (ISS) criterion.


\begin{figure}[t!]
\centering
 \includegraphics[width=0.85\columnwidth]{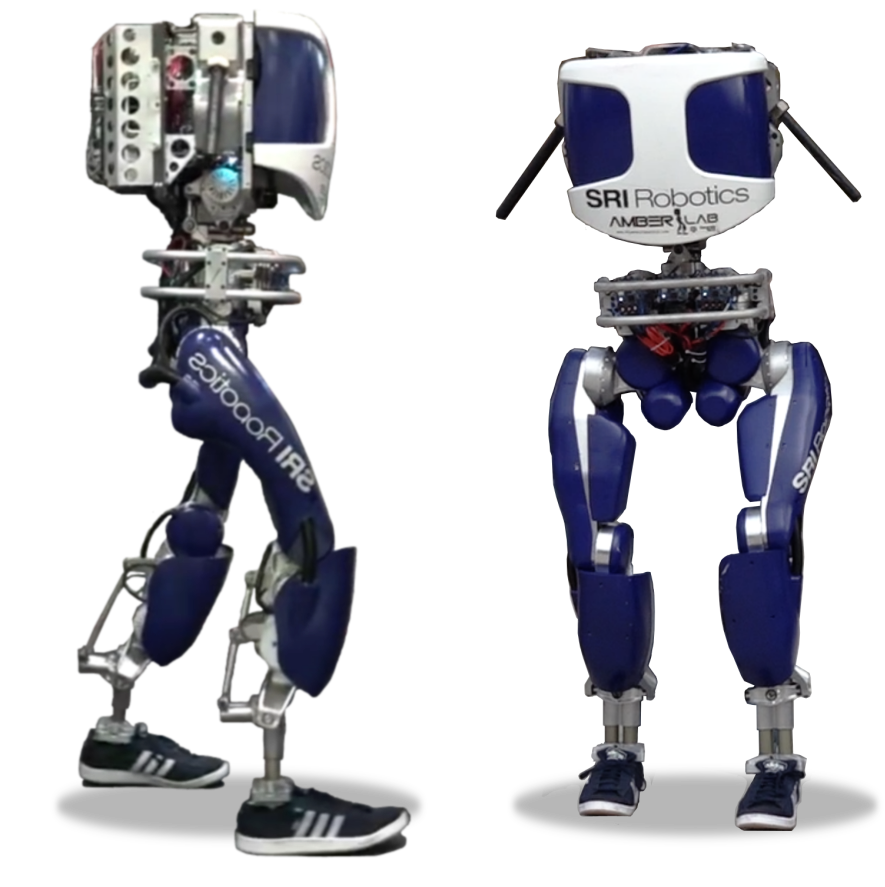}
 \caption{DURUS robot designed by SRI International.}
\label{fig:durusside}
\vspace{-10pt}
\end{figure}

Input to state stability (ISS) theory, mainly developed and popularized by Sontag \cite{sontag2008input} during the 1980's, was a result of 
a growing need for a stronger stability criterion (stronger than BIBO) 
on nonlinear systems.
ISS for hybrid systems was studied only after the mid 2000's \cite{cai2005results,hespanha2005input}, where the focus was on identifying sufficient conditions for stability. 
Bipedal walking robots are a classic example of mechanical hybrid systems involving alternating sequences of continuous (swing) and discrete (footstrike) events. Practical applications of controllers that yield ISS, called input to state stabilizing (ISSing) controllers, for bipedal robots are limited in literature. 
There is work on phase uncertainty to state stability \cite{kolathaya2016time,hscc17running}, which resulted in stable walking and running. 
However, some of the control implementations, like the use of PD controllers for tracking at the low level, were not formally justified or studied.
Therefore, the main contribution of this paper is to establish theoretical results on these planar locomotive controllers, and then use these concepts to realize robust walking behaviors on the bipedal robot DURUS. 
More importantly, this paper formally establishes that exponential stabilization of the continuous dynamics renders the full order hybrid system ISS under a wide variety of modeling and sensing uncertainties.

DURUS is a 23-DOF bipedal robot (\figref{fig:durusside}) with 15 actuators and two linear springs. 
Achieving walking on DURUS is complex due to unmodeled compliance dynamics, underactuation and somewhat sparse sensing abilities - lacking both foot contact force sensors and joint torque sensors. The uncertainties studied in this paper are (but not limited to) modeling uncertainty and phase based uncertainty. To address the model based uncertainty, PD control laws will be used both in simulation and experiment. It is a well known fact that PD control laws render robotic systems input to state stable (or integral-ISS) in the continuous dynamics \cite{7844077,angeli1999input}. 
To address the uncertainty due to the aberrations in the phase variable, which dictates the progression of the gait, we will utilize a time based parameterization of the reference trajectories (see \cite{kolathaya2016time}). With the realization of these controllers experimentally on DURUS, the end result is ISSing control laws that are robust to these two uncertainties.

 Section \ref{sec:iss} contains a brief preliminary on input to state stability. Section \ref{sec:issrobot} contains a brief overview on the ISS properties of robotic systems in the continuous dynamics. 
Section \ref{sec:hs} will introduce the hybrid systems model. Section \ref{sec:hs} will also describe the walking controllers and Section \ref{sec:theory} will describe the main result of the paper: ISS of walking robots. Finally, Section \ref{sec:results} will conclude with simulation and experimental results on DURUS.

\section{Preliminaries on Input to State Stability}\label{sec:iss}

This section will introduce basic definitions and results related to input to state stability (ISS); for a detailed survey on ISS see \cite{sontag2008input}. Most of the content in this section is based on \cite{sontag1989smooth,Sontag89furtherfacts,sontag1991input,sontag1995characterizations}. 


 We consider a general affine nonlinear system in the following fashion:
%
\begin{align}
\label{eq:system3}
\dot{x} =& f(x) + g(x) d,
\end{align}
with $x$ taking values in the Euclidean space $\R^n$ and the input $d \in \R^m$ for some positive integers $n,m$. The mappings $f:\R^n \to \R^n$, $g:\R^n \to \R^{n\times m}$, are Lipschitz functions of $x$, and $f(0)=0$.  
Therefore, the construction is such that for some unknown system $\dot x = \bar f(x)+g(x)u$ a stabilizing controller $u=k(x)$ has been applied. Any deviation from this stabilizing controller can be viewed as $k(x)+d$ such that in \eqref{eq:system3} $f(x) := \bar f(x) + g(x) k(x)$, with $d$ being a new disturbance input. We further assume that $d:\R_{\geq 0} \to \R^m$ takes values in the space of all Lebesgue measurable functions of time: $\|d\|_\infty := \mathrm{sup}_{t \geq 0} \{ | d(t) |\} < \infty$, which can be denoted as $d \in \mathbb{L}^m_\infty$. Here, $|\:.\:|$ is the Euclidean norm.

\newsec{Class $\classK,\classK_\infty$ and $\classKL$ functions.} 
A class $\classK$ function is a function $\alpha:[0,a) \to \R_{\geq 0}$, $a>0$, which is continuous, strictly increasing and satisfies $\alpha(0) = 0$. 
A class $\classK_\infty$ function is a function $\alpha:\R_{\geq 0} \to \R_{\geq 0}$ which is continuous, strictly increasing, proper, and satisfies $\alpha(0) = 0$, 
and a class $\classKL$ function is a function $\beta:\R_{\geq 0} \times \R_{\geq 0} \to \R_{\geq 0}$ such that ${\beta(r, t) \in \classK_\infty}$ for each $t$ and $\beta(r, t) \to 0$ as $t \to \infty$.

%


 We can now define ISS for \eqref{eq:system3}. It is important to note that ISS and related definitions are always w.r.t. the input disturbance $d$. Let $x_0 \in \R^n$ be the initial condition, and let $\phi_t(x_0,d)$ be the solution to the closed loop dynamics of \eqref{eq:system3}: $\dot{\phi}_t(x_0,d) = f(\phi_t(x_0,d)) + g(\phi_t(x_0,d)) d$.  
 We have the following definition of ISS \cite{sontag1989smooth}.
\gap
\begin{definition}\label{def:ISS}{\it
 The system \eqref{eq:system3} is input to state stable (ISS) if there exists $\beta \in \classKL$, and $\iota \in \classK_\infty$ such that
 \begin{align}\label{eq:ISSmaindefinition}
  {|\phi_t(x_0,d)|} \leq \beta(|x_0|,t) + \iota(\|d\|_\infty), & \hspace{10mm} \forall x_0,d, \forall t \geq 0.
 \end{align}
}
\end{definition}
\gap

\begin{definition}{\it
 The system \eqref{eq:system3} is exponential input to state stable (e-ISS) if there exists $\beta \in \classKL$, $\iota \in \classK_\infty$ and a positive constant $\lambda > 0$ such that
 \begin{align}\label{eq:ISSmainexpodefinition}
  {|\phi_t(x_0,d)|} \leq \beta(|x_0|,t) e^{-\lambda t} + \iota(\|d\|_\infty), & \hspace{5mm} \forall x_0,d, \forall t \geq 0.
 \end{align}
}
\end{definition}
\gap

\newsec{Input to state stable Lyapunov functions.} A direct consequence of using ISS concepts is the construction of input to state stable Lyapunov functions (ISS-Lyapunov functions). 

\gap
\begin{definition}\label{def:Lyapunovdefinition}{\it
A smooth function $V:\R^n \to \R_{\geq 0}$ is an ISS-Lyapunov function for \eqref{eq:system3}  if there exist functions $\ubar{\alpha}$, $\bar \alpha$, $\alpha$, $\iota \in \mathcal{K}_\infty$ such that $\forall x,d$
\begin{align}
\label{eq:ISSd}
&\ubar{\alpha}(|x|) \leq V (x)  \leq \bar \alpha (|x|)  \nonumber \\
&\dot{V}(x,d)  \leq - \alpha(|x|) + \iota (\|d\|_\infty). 
\end{align}}
\end{definition}\gap
The inequality condition can be made stricter by using the exponential estimate:
\begin{align}
\label{eq:ISSdstricter}
\dot{V}(x,d)  \leq - c V(x) + \iota (\|d\|_\infty),
\end{align}
which is then called the e-ISS-Lyapunov function. There are also alternate forms like
\begin{align}
\label{eq:ISSdstricteralternate}
\dot{V}(x,d)  \leq - c V(x)  \quad \mathrm{for} \quad |x| \geq \iota (\|d\|_\infty), 
\end{align}
which is also an e-ISS-Lyapunov function.

Instead of the states, if we are interested in the behavior of the outputs of the form $y: \R^n \times \R_{\geq 0} \to \R^k$ that is a function of the states $x$ and time $t$, we have {\it input to output stability}.
\gap
\begin{definition}
 The system \eqref{eq:system3} is input ($d$) to output ($y$) stable (IOS) if there exists $\beta \in \classKL$, and $\iota \in \classK_\infty$ such that
 \begin{align}\label{eq:IOSmaindefinition}
  {|y(\phi_t(x_0,d),t)|} \leq \beta(|x_0|,t) + \iota(\|d\|_\infty), & \hspace{3mm} \forall x_0,d, \forall t \geq 0.
 \end{align} 
\end{definition}
\gap





\newsec{Uncertainty vs. deviation from $k(x)$.} It is important to note that a wide variety of uncertainties can be classified as the deviation from the stabilizing control input $d$. Model parameter uncertainty and phase-based uncertainty were classified in this manner in \cite{kolathaya2016parameter,kolathaya2016time} respectively. Other uncertainties include actuator saturations, unmodeled dynamics appearing in the controllable space, and also noisy sensor feedback. In this paper, we will specifically show how a PD control law can address model uncertainty, and a time based parameterization can address phase uncertainty by viewing them as the input disturbance (deviation) $d$. 

%

\section{Input to state stability of PD controlled robotic systems}\label{sec:issrobot}
In this section, we will describe how to realize linear feedback laws, such as PD control, that render a robotic system ISS. We will study here a robotic system consisting of $n$-DOF and the corresponding configuration space $\Q\subset\R^n$. The configuration $q\in\Q$ consists of $n_j$ joint angles, $n_b$ base DOFs, $n_s$ spring and also $n_s$ dampers. Therefore, we denote the state $x:=(q,\dq)\in T\Q \subset\R^{2n}$, where $n=n_j+n_b+n_s$. We will denote the torque input $u$, which is of dimension $m$. In addition to the torque input, we also have $n_h$ holonomic constraint forces $\Lambda \in \R^{n_h}$ acting at various points on the robot (in DURUS, foot contacts with ground act as holonomic constraints).

\subsection{Dynamics}
Given the states, inputs, and holonomic constraints, the Euler-Lagrangian dynamics is given by:
\begin{align}
  \label{eq:eom-general} 
  D(\q) \ddq + H(\q,\dq) - B u - J_h^T(\q) \Lambda &=0 \nonumber \\
  J_h(\q) \ddq + \dot{J}_h(\q,\dq)\dq &= 0,
\end{align}
where $D(q) \in \R^{n\times n}$ is the positive definite inertia matrix, $H (\q,\dq) \in \R^n$ is the Coriolis-centrifugal-gravity vector, $B\in \R^{n\times m}$ is the one-on-one mapping of the torques to the joints, and $J_h(q)\in\R^{n_h\times n}$ is the Jacobian of the holonomic constraints. We have the following properties of the model (assuming all revolute joints \cite{ROB:ROB2}).
\gap
\begin{property}\label{prop:1}
For some $\ubar{c}_d, \bar c_d > 0$, and $c_c > 0$
\begin{align}
\label{eq:constantscd}
\underbar c_d \leq \|D(\q)\| \leq \bar c_d, \quad \|H(\q,\dq)\| \leq c_c ( 1 + |\dq|^2 ).
\end{align}\end{property}
\gap
\begin{property}\label{prop:2}
 For some $c_h , \ubar{c}_h, \bar c_h >0$
  \begin{align}\label{eq:boundsholonomicconstraints}
  & \|J_h(\q) \| \leq c_h,\quad \|\dot J_h(\q,\dq) \| \leq c_h |\dq|, \\
  & \ubar{c}_h \leq \|J_h(\q) D^{-1}(\q) J_h^T(\q)\| \leq \bar c_h. \nonumber
  \end{align}
%
\end{property}\gap
Here $\|.\|$ denotes the matrix norm, and the matrix $J_h D^{-1} J_h^T$ is invertible by construction\footnote{Here, and throughout much of the paper, the arguments for matrices associated with the dynamics will be suppressed for notational convenience.}. The holonomic constraint forces can be explicitly computed as (see \cite[eqn. (30)]{ames2013towards})
\begin{align}
 \label{eq:holonomicconstraintforces}
  &\Lambda  = - \Lambda_1 B u +\Lambda_1 H - \Lambda_2 \\
 &\Lambda_1  = ( J_h D^{-1} J_h^T)^{-1}  J_h D^{-1}, \quad  \Lambda_2  = ( J_h D^{-1} J_h^T)^{-1} \dot J_h \dq. \nonumber
\end{align}
Thus \eqref{eq:holonomicconstraintforces} can be substituted in the dynamics to yield 
\begin{align}
 \label{eq:dynamicscombined}
  D \ddq + (\1 - J_h^T \Lambda_1) H + J_h^T\Lambda_2 = (\1 -J_h^T\Lambda_1) B u,
\end{align}
where $\1$ is an identity matrix of appropriate dimension. \eqref{eq:dynamicscombined} can be represented in statespace
statespace form 
\begin{align}
\label{eq:dynamicsmech}
\left [ \begin{array}{c} \dq \\ \ddq \end{array} \right ] &= f (q,\dq) + g (q,\dq) u, 
\end{align}
which is similar to \eqref{eq:system3}. 
\subsection{PD Control}
There has been work on proving the input to state stability, specifically integral input to state stability (iISS), of these kinds of robotic systems \cite{angeli1999input,7844077}. We will use some of these ideas to realize controllers for walking. 

If we assume full actuation ($B$ to be square and full rank, $n=m$), we can choose the following particular control law (linear feedback law):
\begin{align}
 \label{eq:PDISS}
u = -K_p (\q - \q_d) - K_d (\dq - \dq_d), 
\end{align}
where $K_p, K_d$ are constant gain matrices of appropriate dimension ($n$ for $B$ full rank), and $q_d,\dq_d$ are desired joint angles and velocities. These desired values are either constants or functions of states and time: $\q_d : \Q \times \R_{\geq 0} \to \Q$, and $\dot \q_d : T_q\Q \times \R_{\geq 0} \to T_q\Q$. Note that $q_d$, $\dq_d$ can also be pure functions of time, and are particularly used for the walking control of DURUS (more details in Section \ref{subsec:timebasedoutputdynamics}. 
One way of characterizing ISS is shown via the following example.
\gap
\begin{example}
 By substituting \eqref{eq:PDISS} in \eqref{eq:dynamicscombined}:
\begin{align}
 &D (\ddq - \ddq_d)  = (\1 - J_h^T \Lambda_1)(- B K_p (\q - \q_d) - B K_d (\dq - \dq_d)  ) + d  \nonumber\\
 &{\rm{where}}\quad d = - (\1 - J_h^T \Lambda_1) H - J_h^T\Lambda_2 - D \ddq_d ,
\end{align}
where a part of the model itself is regarded as the disturbance. Denote $y(\q,t):=\q - \q_d(t) $. Therefore
\begin{align}
\label{eq:outputsdynamicsexample1}
  \ddot y  - D^{-1}(\1 - J_h^T \Lambda_1) B K_p \dot y - D^{-1}(\1 - J_h^T \Lambda_1) B K_d y = D^{-1}d.
\end{align}
Since the inertia matrix $D$ is positive definite, the signs of the errors $y,\dot y$ do not change. \eqref{eq:outputsdynamicsexample1} can be written in statespace form:
\begin{align}
\begin{bmatrix}  \dot y \\ \ddot y
 \end{bmatrix} &= \begin{bmatrix} 0 & \1 \\  D^{-1}(\1 - J_h^T \Lambda_1) B K_p &  D^{-1}(\1 - J_h^T \Lambda_1) B K_d \end{bmatrix} \begin{bmatrix}
		y \\ \dot y \end{bmatrix} \nonumber \\
		&\qquad																+ \begin{bmatrix}
																					    \0 \\ D^{-1} \end{bmatrix}d,
\end{align}
which is of the form \eqref{eq:system3}. We can establish ISS by tuning the gains $K_p$, $K_d$ appropriately. This type of characterization increases the gap between the assumed and the actual robot model, but can yield smaller tracking errors based on the nature of $q_d$, and the gains $K_p,K_d$. 
\end{example}
\gap

When $B$ is not full rank ($m<n$), we can generalize this tracking problem in terms of the difference between $k=m$ actual and desired values: $y (\q) = y_a(\q) - y_d (\q) \in \R^m$. With the desired outputs, we can obtained the desired configuration and velocities of the robot via an inverse mapping (for example $\Phi^{-1}: \R^{n-m}\times \R^m \to \Q$) to obtain
\begin{align}
	\q_d(q) &= \Phi^{-1}(q_C,y_d(q)), \nonumber \\
	 \dq_d(q,\dq) &= d\Phi^{-1}(q_C,y_d(q)) \dq,
\end{align}
where $q_C$ is the configuration of unactuated joints of the robot. In the context of bipedal walking, this is formally called PHZD reconstruction, and is explained more in Section \ref{sec:PHZDrecons}. The following example illustrates how PD controllers can be realized for these types of outputs.\gap
\begin{example}
We have the following output dynamics:
\begin{align}\label{eq:outputdynamics}
 \dot y &= L_f y \nonumber \\
 \ddot y &= L^2_f y + L_g L_f y u 
\end{align}
where
\begin{align}
 L_f y(q,\dq) &= \dot y_a(\q,\dq) - \dot y_d(\q,\dq) \nonumber \\
 L_f^2 y(q,\dq)  &= \left [ \frac{\partial L_f y (\q,\dq)}{\partial \q} \:\: \frac{\partial L_f y(\q,\dq)}{\partial \dq} \right ] f(\q,\dq) \nonumber \\
 L_g L_f y(q,\dq)  &=  \left [ \frac{\partial L_f y(\q,\dq)}{\partial \q} \:\: \frac{\partial L_f y(\q,\dq)}{\partial \dq} \right ] g(\q,\dq),
\end{align}
where $L_f,L_g$ are the Lie derivatives. 
Substituting a PD control law similar to \eqref{eq:PDISS} we have
\begin{align}\label{eq:outputdynamics2}
 \ddot y &= L^2_f y + L_g L_f y u \nonumber \\
&= L^2_f y + L_g L_f y B^T ( -K_p (q - \q_d) - K_d (\dq - \dq_d ) ),
\end{align}
where $B^T$ is used to match the dimensions.
By appropriately tuning $K_p, K_d$ and through the selection of $y_d$, it is possible to realize input to output stability (IOS) of this system. Since $m<n$, the output dynamics do not fully represent the full dynamics of the robot. Therefore, for underactuated systems, ISS is achieved when both the output dynamics and the passive (``unactuated'' states) dynamics are included in the analysis (more on this in Section \ref{sec:theory}). 
\end{example}
\gap

 Asymptotic stability (and even exponential stability) of PD control of continuous robotic systems has been extensively studied in literature \cite{1643376,4790013,koditschek-the_robotics_review-1989,371478,doi:10.5772/61537}. Similarly, ISS of PD control of robotic systems has been studied in \cite{angeli1999input}, and there are several forms of the disturbance dynamics that can be viewed and analyzed\footnote{A result of particular importance is e-ISS of PD tracking of robotic systems, for which a sketch of the proof is provided in Appendix \ref{app:eisspdcontrol}. Also see \cite{wen1988new} for a detailed survey on PD based controllers in robotic systems.}.
 In this paper, we will consider a more traditional form of disturbance $d$ obtained from \cite{Sontag89furtherfacts}, i.e., deviation from a stabilizing control input. For the output dynamics of the form \eqref{eq:outputdynamics}, we know that a suitable controller that yields stability is a feedback linearizing controller:
 \begin{equation}
 	u_{\rm{IO}} = L_{g} L_{f} y ^{-1} \left ( - L_{f}^2 y - 2\epsilon L_f y - \epsilon^2 y \right ), \quad \epsilon > 0, \label{eq:fblin}
 \end{equation}
which when added and subtracted in \eqref{eq:outputdynamics} yields
 \begin{align}\label{eq:dynamicsreference}
  \ddot y &= L^2_f y + L_g L_f y  ( u  + u_{\rm{IO}} - u_{\rm{IO}})  \\
          &= - 2\epsilon \dot y - \epsilon^2 y  + L_g L_f y  \underbrace{( u  - u_{\rm{IO}} )}_{=:d}. \nonumber
 \end{align}
 The disturbance input $d$, in this context, is effectively the deviation from the feedback linearizing controller \eqref{eq:fblin}. 

\begin{figure}[t!]\centering
\includegraphics[width=0.42\columnwidth]{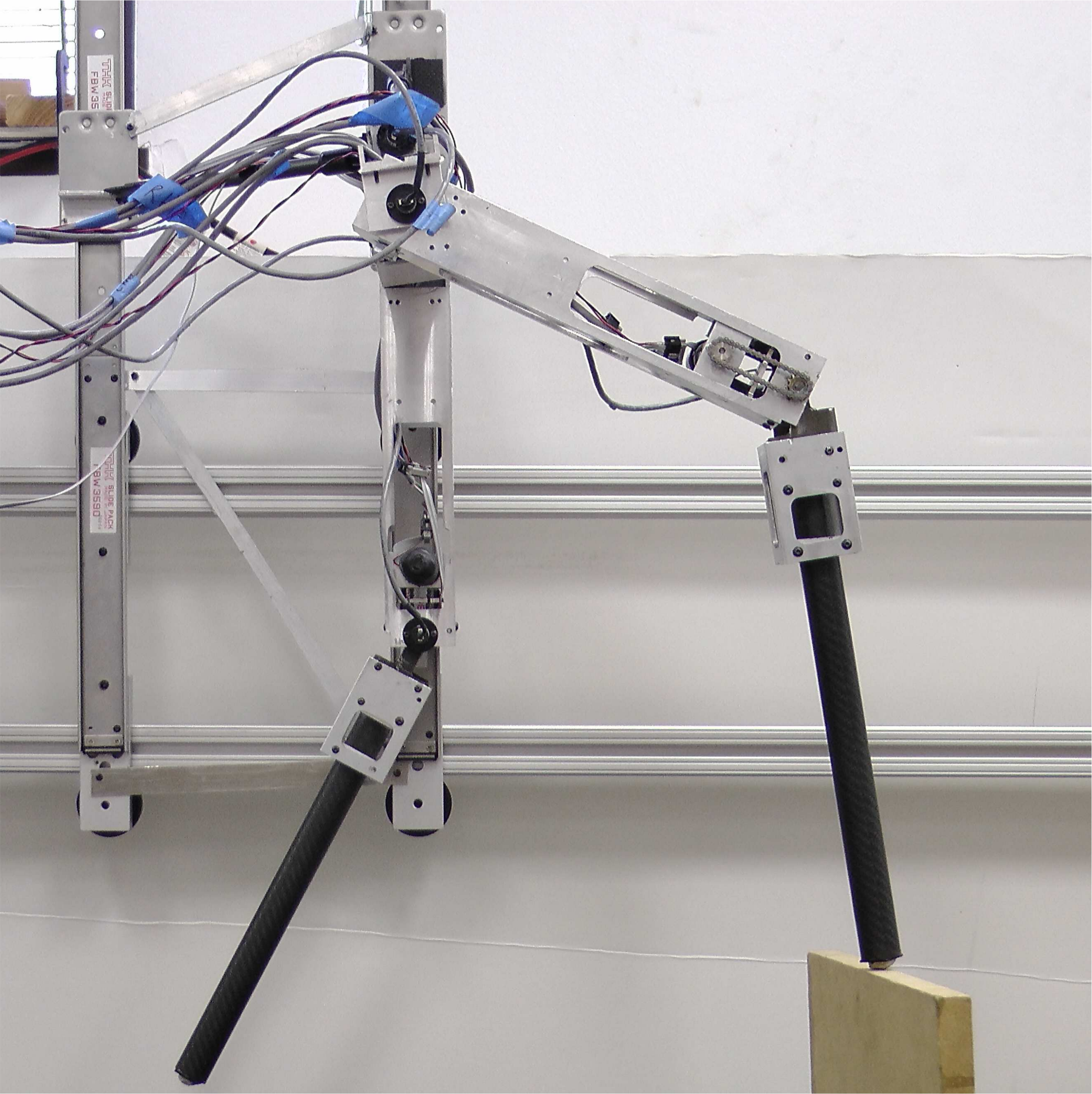}
\includegraphics[width=0.42\columnwidth]{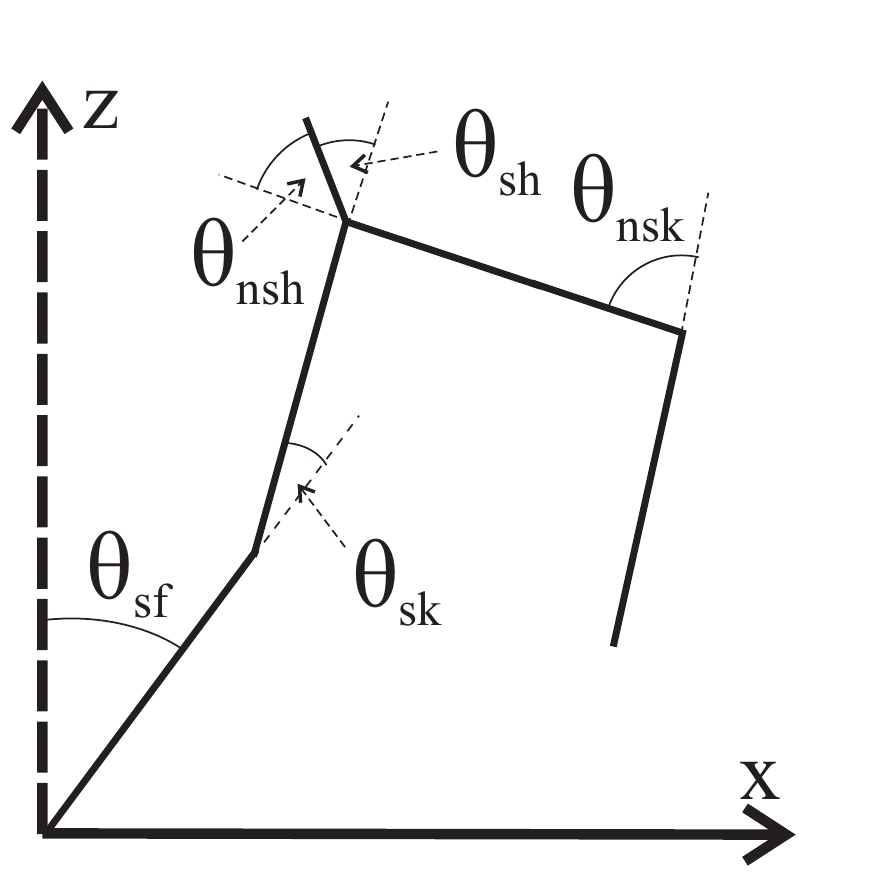}
\caption{AMBER is shown on the left, with the configuration of the 5 joints shown on the right.}
 \label{fig:amber}
\end{figure}
\gap
\begin{remark}
The two controllers \eqref{eq:PDISS}, \eqref{eq:fblin} are indeed very different, and the resulting deviation could be large. Since the focus is on realizing controllers that yield a low $d$ to $y$ gain, this type of characterization is acceptable as long as the deviation remains bounded. 
To illustrate, we will take into consideration a fully actuated 5-DOF serial chain manipulator, AMBER\footnote{AMBER was originally designed for walking on a treadmill \cite{Yadukumar2013}.}, as shown in \figref{fig:amber}, and drive the joint configuration $q=(\theta_{sf},\theta_{sk},\theta_{sh},\theta_{nsh},\theta_{nsk})\in \R^5$ to a desired configuration by applying a linear feedback law \eqref{eq:PDISS}. We have the set of actual outputs $y_a(q) := q $, and the desired outputs $y_d(q) := 0$, i.e., the goal is to drive $y(q) =  q$ to zero from an arbitrary configuration. Applying the linear feedback law \eqref{eq:PDISS}, indeed, results in convergence of the joint angles to zero (as shown in \figref{fig:amber1result}). Note that the immediate choice for stabilization would have been \eqref{eq:fblin}, but instead, a model-free controller, \eqref{eq:PDISS}, was applied resulting in low tracking errors ($< 0.01$ $\rm{rad}$).
Explicit computation of the upper bounds on the tracking errors as a function of $d$ can be obtained from the following Lemma.
\end{remark}

\begin{figure}[t!]
\includegraphics[width=0.49\columnwidth]{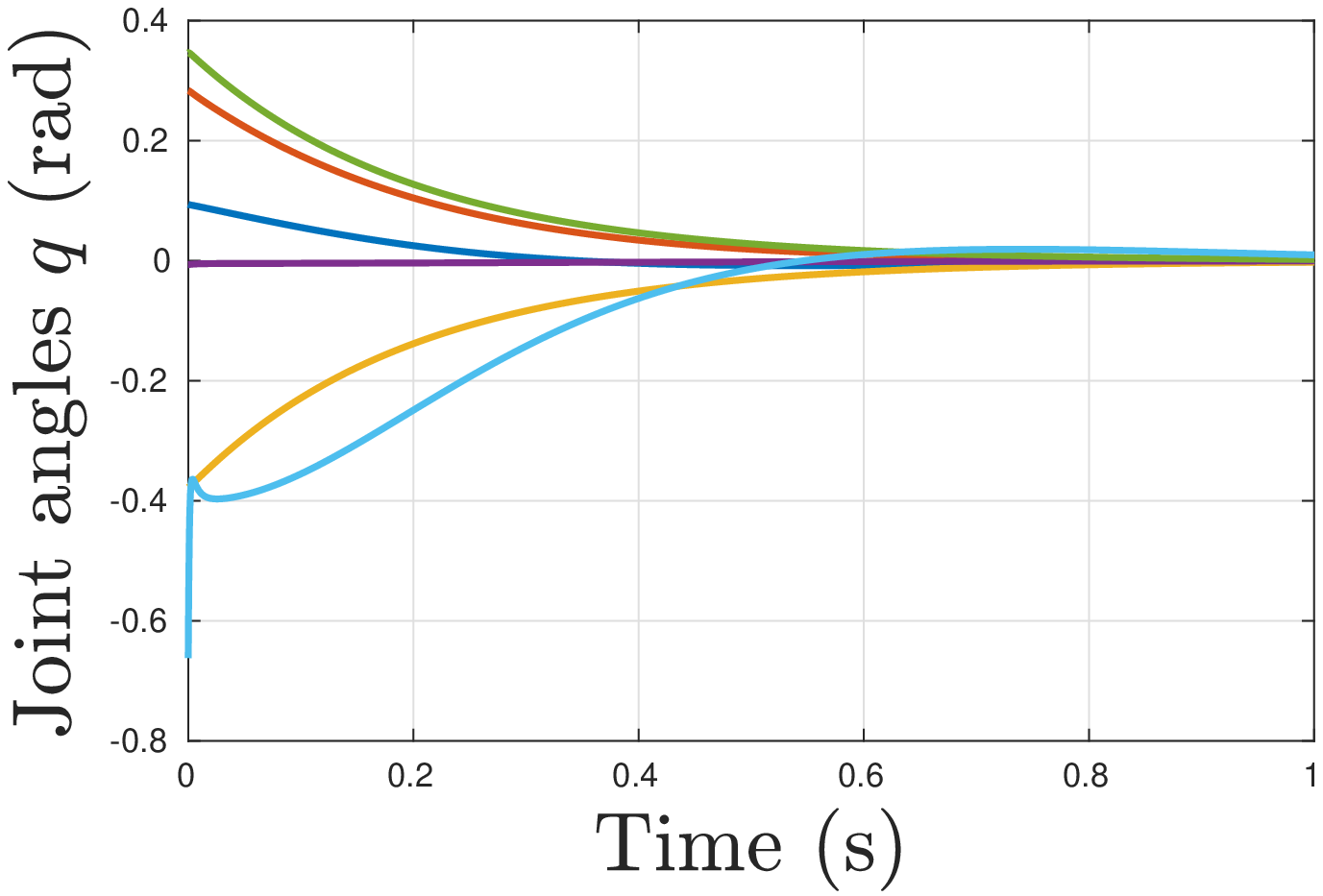} 
\includegraphics[width=0.49\columnwidth]{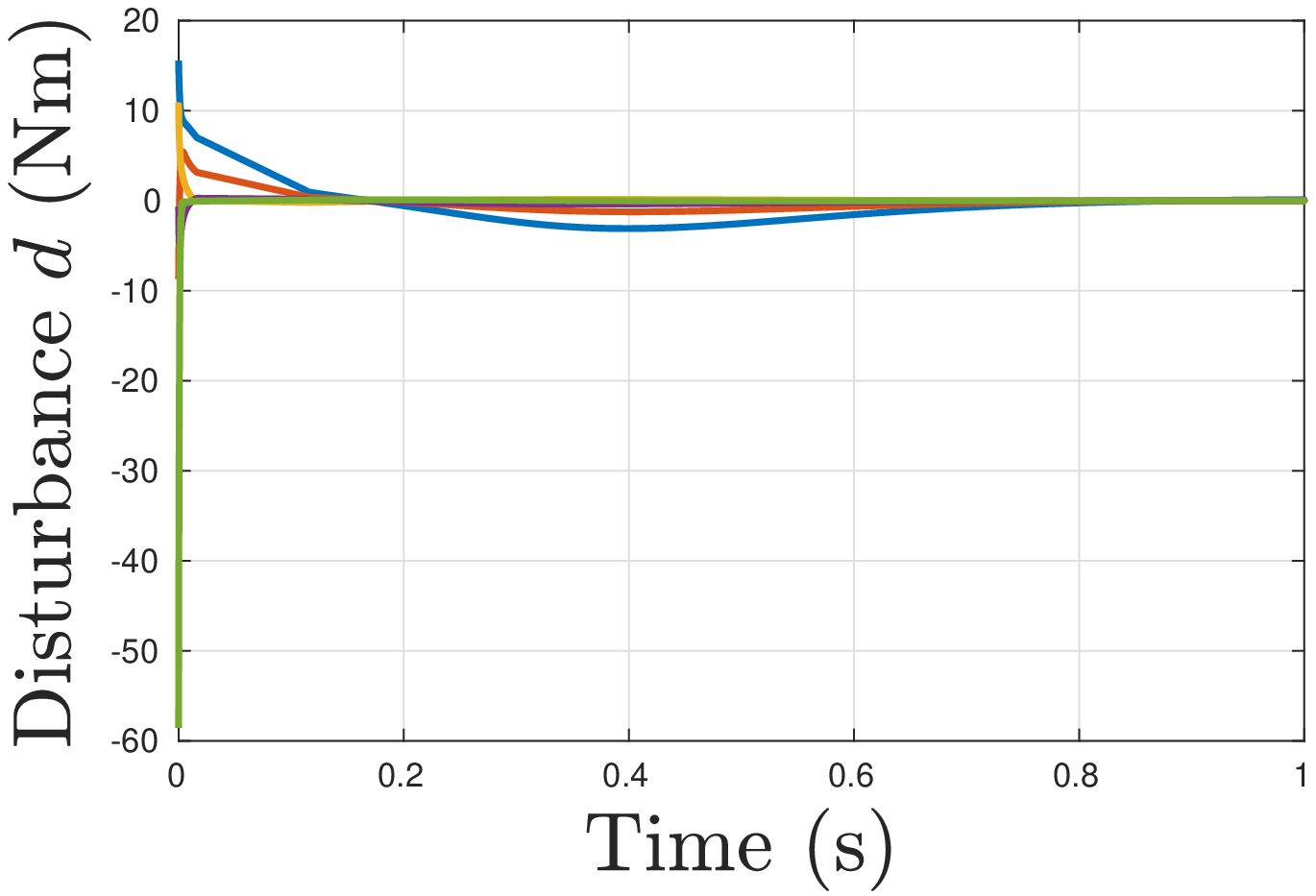} 
\caption{The tracking errors of the 5 joints are shown on the left, with the input deviations on the right. It can be observed that for input deviations as large as $ 60$ $ \rm{Nm}$, the tracking error is bounded. In this example, $d$ not only remained bounded, but also reduced over time. This behavior is, of course, not always true and requires manual tuning of the gains and output parameters of the system.}
 \label{fig:amber1result}
\end{figure}

\gap 
\begin{lemma}{\it
 \label{lm:eissofcontinuous}
 The transverse dynamics of the form \eqref{eq:dynamicsreference} is exponential $d$ to $y$ stable. }
\end{lemma}
\gap
\begin{proof}
The outputs $y$ chosen are relative degree two (twice differentiable to yield the control input). Therefore denote: $\eta_2 := [y^T,\dot y^T]^T$. 
 \eqref{eq:dynamicsreference} can be written in statespace form as
 \begin{align}\label{eq:eta2dynamicsdescription}
\dot \eta_2 = \underbrace{  \begin{bmatrix}
                               \0 & \1 \\
                               - \epsilon^2 \1 & - 2 \epsilon \1
                              \end{bmatrix}  }_{A_2} \eta_2 + \underbrace{\begin{bmatrix}
				    0\\
				    L_{g} L_{f} y
				    \end{bmatrix}}_{B_2}  d,
\end{align}
where $A_2$ is Hurwitz (see \cite[eqn. (11)]{TAC:amesCLF}), and $B_2(\q)$ depends on $\q$ only. $\1$ is an identity matrix of appropriate dimension. This can be viewed in terms of the following candidate ISS-Lyapunov function:
 \begin{align}
  V_{\eta_2} (\eta_2) : = \eta_2^T P_2 \eta_2,  
 \end{align}
 where $P_2$ is the solution to the Lyapunov equation $A^T_2 P_2 + P_2 A_2 = - Q_2$, $Q_2 > 0$. The subscript indicates that the constants are defined for the relative degree two outputs. The derivative of $V_{\eta_2}$ yields 
 \begin{align}\label{eq:lyapunovmainderivative}
  \dot V_{\eta_2} &= \eta_2^T (A^T_2 P_2 + P_2 A_2 ) \eta_2 + 2 \eta^T_2 P_2 B_2 d  \nonumber \\
	      &= - \eta_2^T Q_2 \eta_2  + 2 \eta^T_2 P_2 B_2 d .
\end{align}
Let $\gamma_1 , \gamma_2 > 0$ such that $( \gamma_1 + \gamma_2 ) P_2 \leq Q_2$ (see \cite[eqn. (46)]{kolathaya2016parameter}). We therefore have
\begin{align}\label{eq:lyapunovmainderivative2}
\dot V_{\eta_2} \leq - \gamma_1 V_{\eta_2} \quad {\rm{for}}   \quad \gamma_2 V_{\eta_2} \geq   2  \eta^T_2  P_2 B_2 d,
\end{align}
 which implies that $V_{\eta_2}$ is decreasing exponentially:
 \begin{align}\label{eq:etaboundmaadi}
\dot V_{\eta_2} \leq - \gamma_1 V_{\eta_2} \quad {\rm{for}} \quad |\eta_{2} | \geq \frac{2 \lambda_{\max}(P_2)}{\gamma_2 \lambda_{\min}(P_2)} \|B_2  \| \|d\|_\infty,
 \end{align}
 which is of the form \eqref{eq:ISSdstricteralternate}. $\lambda_{\min}(.)$, $\lambda_{\max}(.)$ provide the minimum and maximum eigenvalues of the positive definite matrix $P_2$. $\|B_2\|$ is the matrix norm, which has an upper bound:
 \begin{align}
  \|B_2\|  &= \|L_g L_f y \| = \left \|\frac{\partial y }{\partial q} D^{-1} (\1 - J_h^T\Lambda_1 ) B \right \| \nonumber \\
	  &\leq c_y \underbar c_d^{-1} ( 1 + c_h^2 \underbar c_h^{-1}),
 \end{align}
 where the constants $c_d, c_h, \underbar c_h$ are obtained from \eqref{eq:constantscd}, \eqref{eq:boundsholonomicconstraints}, and $c_y = \sup_{q}\|\frac{\partial y }{\partial q}\|$. 
%
\end{proof}
\gap
Boundedness of $B_2$ is established due to the fact that the control system is affine \eqref{eq:system3}, and both $\|\frac{\partial y }{\partial q}\|$, $g(x)$ are bounded. Note that the Jacobian of the outputs is bounded by construction. We will treat the boundedness of $B_2$ as an assumption for future reference.
\gap
\begin{assumption}\label{assumption:1}
 The output dynamics of the form \eqref{eq:outputdynamics} is globally Lipschitz w.r.t. the control input $u$ (and correspondingly the disturbance input $d$).
\end{assumption}
\gap
\begin{remark}
The input deviation $d$ can be expanded as follows:
\begin{align}\label{eq:disturbancereference}
 | d | & = |- B^T K_p (q - \q_d) - B^T K_d (\dq - \dq_d)  - u_{\rm{IO}}|   \\
    & \leq |  - B^T K_p (q - \q_d) - B^T K_d (\dq - \dq_d) \nonumber \\
     & \qquad \qquad + L_g L_f y^{-1} (\epsilon^2 y + 2\epsilon \dot y ) | + |L_g L_f y^{-1} L_f^2 y | .\nonumber 
\end{align}
As mentioned previously, the input disturbance $d$ is dependent on the states of the system, and the ISS property will not hold when the outputs $y$, $\dot y$ get larger in time. On the other hand, if these values evolve such that the resulting deviation $d$ is bounded\footnote{Boundedness of the states (and hence of the outputs) for continuous robotic systems is well known \cite[Table $1$]{wen1988new}.}, the resulting output dynamics are rendered ISS. We make the following assumption:
\gap
\begin{assumption}
 The disturbance effects can be minimized by manual tuning of the gains $K_p$, $K_d$. In other words, for every $\delta>0$, there exist constants $K_p$, $K_d$ such that $\|d\|_\infty \leq \delta$, where $d$ is defined as in \eqref{eq:disturbancereference}.
\end{assumption}
\gap
This is evident for DURUS from \figref{fig:resultssimexp} and \figref{fig:resultsdist}, where for input deviations as large as $100$ Nm, the tracking errors are observed to not exceed $0.05-0.06$ rad. To justify these results on DURUS, it is necessary to extend the concepts presented in this section to hybrid systems.
\end{remark}

\section{Robot Walking Model and Control}\label{sec:hs}

In this section, we will discuss the hybrid model of DURUS walking. DURUS is an underactuated $23$-DOF bipedal robot designed by collaboratively between SRI International, Dynamic Robotics Lab and AMBER Lab (see \figref{fig:durusside}) with $15$ actuators and $2$ springs. Different behaviors including multi-contact walking \cite{reheralgorithmic} were successfully realized in DURUS. In this paper, we will focus on flat-footed walking only. More complex behaviors will be studied in future.

\subsection{Hybrid walking model} The DURUS walking model has two continuous events, double support ($\dos$) and single support ($\sis$), and two discrete events, lift-off and foot-strike, that alternate between each other. We, therefore, have a directed graph, $\DirectedGraph = (\Vertex,\Edge)$, with the set of vertices, $\Vertex = \{ \dos, \sis \}$, representing the continuous events and the set of edges, $ \Edge = \{ (\dos,\sis), (\sis,\dos) \} \subset \Vertex \times \Vertex$, representing the discrete events. A pictorial representation of these individual events and the switch between them are shown in \figref{fig:hs}. For the single support phase, the stance foot interacts with the ground resulting in ground reaction forces acting on the stance leg. Similarly, for the double support phase, both the feet interact with ground resulting in the forces acting on both legs. 
These forces are enforced via holonomic constraints. The continuous dynamics is represented by \eqref{eq:dynamicscombined} with the forces $\Lambda_\vi$ and Jacobians $J_\vi$ now dependent on the phase $\vi\in\Vertex$.

 \begin{figure}
\centering
	\includegraphics[width=0.7\columnwidth]{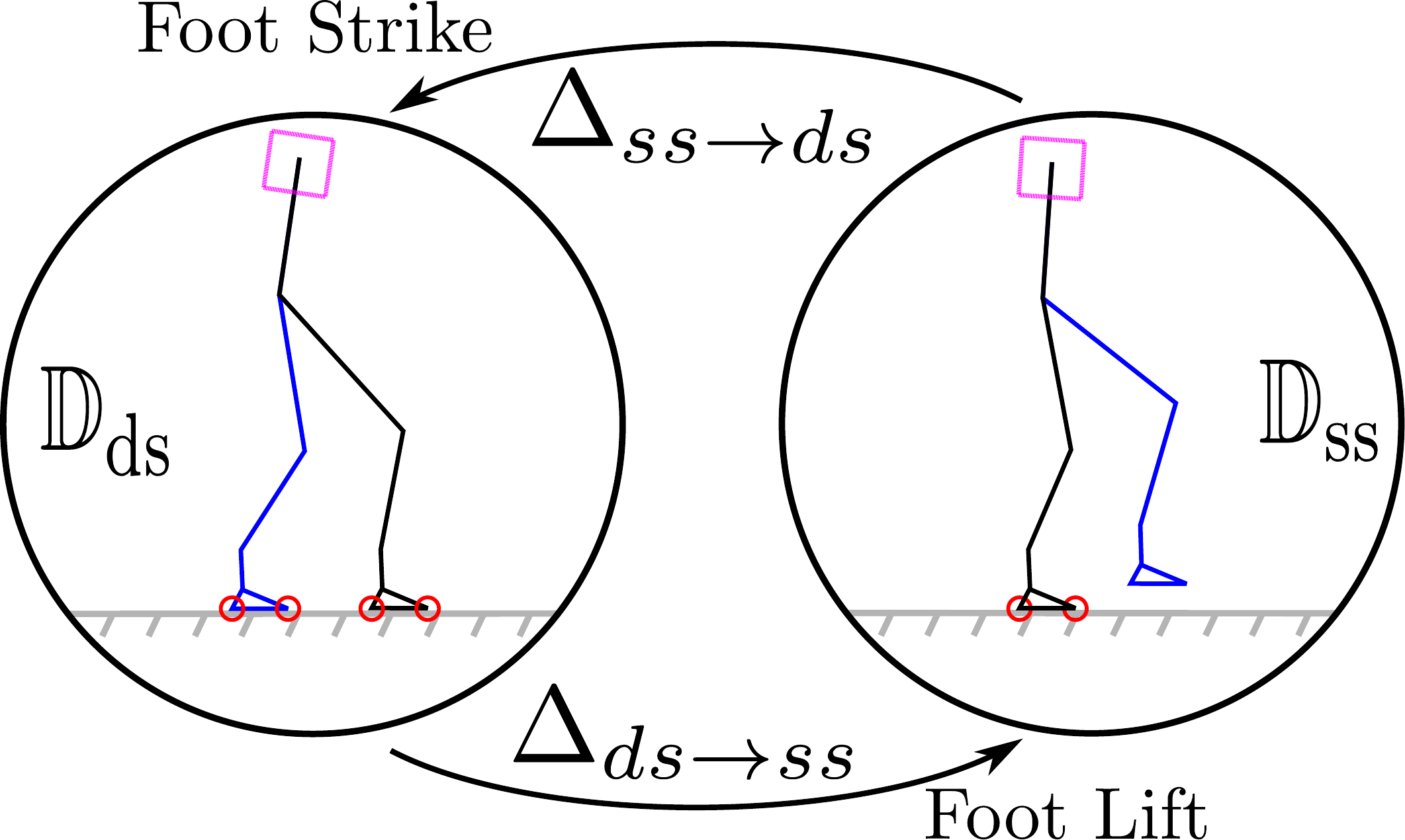}
	\caption{Hybrid system model for the walking robot DURUS.}
	\label{fig:hs}
\end{figure}

\newsec{Hybrid control system.}
 \label{def:genhs}
The hybrid control system model of DURUS consists of a directed cycle: $\DirectedGraph = (\Vertex,  \Edge)$, the set of inputs $\ControlInput=\{\ControlInput_{\dos}, \ControlInput_{\sis}\} $, the set of domains $\Domain=\{\Domain_{\dos}, \Domain_{\sis} \}$, the set of guards $\Guard = \{ \Guard_{\dos},  \Guard_{\sis} \}$, the set of switching functions $\ResetMap=\{ \ResetMap_{(\dos,\sis)},  \ResetMap_{(\sis,\dos)} \}$, and the set of fields $\mathbb{FG}=\{(f_{\dos},g_{\dos}), (f_{\sis},g_{\sis}) \}$. These form a tuple:
\begin{align}
  \label{eq:hybrid-control-system}
  \HybridControlSystem = (\DirectedGraph, \ControlInput, \Domain, \Guard, \ResetMap, \mathbb{FG}).
\end{align}
Note that $\ControlInput_\vi \subset \R^m$, $\Domain_\vi \subset T\Q \times \ControlInput_\vi$, for $\vi\in\Vertex$. For each domain, we have a guard set of co-dimension $1$ $\rm{:}$ $\Guard_{\dos} \subset \Domain_{\dos}$,  $\Guard_{\sis} \subset \Domain_{\sis}$. Denote the projection of the domain and guard sets to the states (only) as $\Guard_\vi |_x, \Domain_\vi |_x $ respectively. $\ResetMap$ consists of the set of switching functions that maps from the guard of one domain to the next domain. For example, $\ResetMap_{(\dos,\sis)}:\Guard_{\dos} |_x  \to \Domain_{{\sis}} |_x$, where $\Guard_{\dos} |_x \subset \Domain_{\dos}|_x$. 
$\ResetMap_{(\dos,\sis)}$ is an identity map applied when the foot leaves the ground to enter the single support phase, while 
$\ResetMap_{(\sis ,\dos)}$
contains the impact equations occurring when the foot strikes at the end of the step. $(f_\vi,g_\vi)$ are given by \eqref{eq:dynamicscombined} and \eqref{eq:dynamicsmech}. See \cite{WGCCM07} for more details.

\subsection{Outputs for walking}
\label{sec:control}

The goal of this section is derive controllers that realize a walking gait in the bipedal robot. 
In order to achieve robotic walking, a periodic orbit is constructed (gait design) and a suitable controller is applied that tracks this reference periodic orbit. 
We have the set of actual outputs of the robot as $y^a:T\ConfigSpace \to \mathbb{R}^k$, and the desired outputs as $y^d :\R_{\geq 0} \to \mathbb{R}^k$. $y^d$ is parameterized by a phase (or time) variable $\tau:T\ConfigSpace \to \R_{\geq 0}$ (or $\tau:\R_{\geq 0} \to \R_{\geq 0}$ for time based). These outputs and their dimensions depend on the domain subscript $\vi\in\Vertex$. We have the relative degree one outputs (velocity outputs)
\begin{align}\label{eq:d1o}
y_{1,\vi} (q,\dq) = y^a_{1,\vi}(q,\dq) - y^d_{1,\vi}(\alpha_\vi) \in \R^{k_{1,\vi}},
\end{align}
and the relative degree two outputs (pose outputs)
\begin{align}
\label{eq:d2o}
y_{2,\vi} (q) = y^a_{2,\vi}(q) - y^d_{2,\vi}(\tau_\vi,\alpha_\vi) \in \R^{k_{2,\vi}},
\end{align}
with the subscript $\vi\in\Vertex$ denoting the domain, $\alpha_\vi$ denoting the parameters of the desired trajectory. ${k_{1,\vi}}+{k_{2,\vi}}=k_\vi$. These outputs are called \emph{virtual constraints} in \cite{WGCCM07}. The phase variable, $\tau_\vi$, for relative degree two outputs is typically a function of the configuration $\tau_\vi(q)$. Walking gaits, viewed as a set of desired periodic trajectories, are modulated as functions of a phase variable to eliminate the dependence on time \cite{kolathaya2016time}. 

\subsection{State based output dynamics} If the phase variable $\tau_\vi$ is state dependent ($\tau_\vi(\q)$), we have the following output dynamics:
\begin{align}
\dot y_{1,\vi} &= L_{f_\vi} y_{1,\vi}  + L_{g_\vi} y_{1,\vi} u \nonumber \\
\ddot y_{2,\vi} &= L_{f_\vi}^2 y_{2,\vi} + L_{g_\vi} L_{f_\vi} y_{2,\vi} u,
\end{align}
where $L_{f_\vi},L_{g_\vi}$ denote the Lie derivatives. These outputs are also called the transverse coordinates. Therefore, with $k_{1,\vi}$ relative degree one outputs and $k_{2,\vi}$ relative degree two outputs, we denote the transverse coordinates as
 $\eta_\vi = \begin{bmatrix} y_{1,\vi}^T , y_{2,\vi}^T,  \dot y_{2,\vi}^T \end{bmatrix}^T \in \R^{k_{1,\vi}+2k_{2,\vi}}$.
In order to track the state based outputs of the system, we can employ feedback linearization
\begin{align}
\label{eq:fblinm}
 u_{\rm{IO}} = \begin{bmatrix}L_{g_\vi} y_{1,\vi} \\  L_{g_\vi} L_{f_\vi} y_{2,\vi}\end{bmatrix} ^{-1} \left ( - \begin{bmatrix} L_{f_\vi} y_{1,\vi} \\ L_{f_\vi}^2 y_{2,\vi} \end{bmatrix} - \begin{bmatrix}  \epsilon y_{1,\vi} \\ 2\epsilon L_fy_{2,\vi} + \epsilon^2 y_{2,\vi} \end{bmatrix} \right ),
\end{align}
that results in the outputs going to zero exponentially. In robots like DURUS, where underactuations are frequently observed in every step, the number of outputs is less than the DOF of the robot, resulting in the dynamics of the coordinates that are normal to the transverse coordinates. These coordinates can be mathematically constructed to what are called zero dynamic coordinates, $z_\vi$. More details are given below.

\newsec{Zero dynamics.}
When the control objective is met such that $\eta_\vi = 0$ for all time then the system is said to be operating on the {\it  zero dynamics surface} \cite{TAC:amesCLF}. 
Further, by relaxing the zeroing of the velocity outputs, we can realize {\it partial zero dynamics surface} \cite{ames2014human}:
\begin{align}
\label{eq:zerodyn2}
 \PZD{\vi} = \{ (q,\dq) \in \Domain_{\vi}|_x : y_{2,\vi} = 0 , L_{f_\vi} y_{2,\vi} = 0 \}.
\end{align}
The humanoid DURUS has feet and employs ankle actuation to propel the hip forward during the continuous dynamics. Thus, a corresponding relative degree $1$ output is used, resulting in \textit{partial zero dynamics} \cite{ames2014human}. 

\subsection{Hybrid zero dynamics}
Any domain specific tracking controller guarantees partial zero dynamics only in the continuous dynamics. For hybrid systems, we use the notion of hybrid zero dynamics and partial hybrid zero dynamics. Therefore, for a hybrid control system $\HybridControlSystem$, {\it partial hybrid zero dynamics} (PHZD) can be guaranteed if and only if the discrete maps $\ResetMap_\ei$ are invariant of the partial zero dynamics in each domain. 
As a result, the parameters $\alpha_\vi$ of the outputs must be chosen in a way which renders the surface invariant through impacts: 
\begin{align}\label{eq:resetmapvW}
 \ResetMap_\ei ( \mathbb{PZ}_{{\rm{source}}(\ei)} \cap \Guard_{\ei}|_x ) \subset \mathbb{PZ}_{{\rm{target}}(\ei)},
\end{align}
where $\ei = ({\rm{source}}(\ei),{\rm{target}}(\ei))$ is the pair containing the source and target vertices for each edge.
\figref{fig:pzd} depicts the dynamics of a two domain hybrid system with 2-dimensional partial zero dynamics. The best way to ensure hybrid invariance under a discrete transition is through the careful selection of the desired trajectories (desired gait) via the parameterization: $\alpha_\vi$. Hence if the desired trajectories are a function of B\'ezier polynomials, the parameters $\alpha_\vi$ are the coefficients. These coefficients are chosen by using a direct collocation based walking gait offline optimization problem, which is explained in \cite{Hereid_etal_2016}.

The zero dynamics are characterized by the zero dynamic coordinates $z_\vi \in \R^{2n-k_{1,\vi}-2k_{2,\vi}}$, which when combined with the transverse coordinates $\eta_\vi$ form the transformed statespace with the following dynamics: 
\begin{align}
\label{eq:dynamicseta}
 \dot{\eta}_\vi  &= \begin{bmatrix}
                     L_{f_\vi} y_{1,\vi}    \\ L_{f_\vi} y_{2,\vi} \\
                     L_{f_\vi}^2 y_{2,\vi} 
                    \end{bmatrix} + \begin{bmatrix}
				    L_{g_\vi} y_{1,\vi} \\ 0\\
				    L_{g_\vi} L_{f_\vi} y_{2,\vi}
				    \end{bmatrix} u \nonumber \\
 \dot{z}_\vi    &= \Psi_\vi(\eta_\vi,z_\vi), \:\: \vi \in \Vertex
\end{align}
 When the transverse and the zero dynamics are combined it results in the full order dynamics. Based on this construction, we have the diffeomorphism $\Phi_\vi : \pi_x(\Domain_\vi) \to \R^{2n}$ that maps from $x = (\q,\dq)$ to $(\eta_\vi,z_\vi)$. This diffeomorphism can be divided into parts:
 \begin{align}
  \label{eq:diffeomorphism}
  \Phi_\vi (x) &=  \left[ \begin{array}{c} 
   \Phi_{1,\vi}(x) \\ \hdashline \Phi_{2,\vi}(x) \\ \hdashline \Phi_{3,\vi}(x)
  \end{array} \right] = \left[ \begin{array}{c}
			y_{1,\vi} (\q,\dq) \\ \hdashline y_{2,\vi} (\q) \\ \dot y_{2,\vi} (\q,\dq) \\ \hdashline z_\vi (\q,\dq)
		    \end{array} \right]     \\
  \Phi^\PZ_\vi (x) &=  \left[ \begin{array}{c} 
   \Phi_{1,\vi}(x)  \\ \hdashline \Phi_{3,\vi}(x)
  \end{array} \right]    
		    .
 \end{align}
 Similarly, the outputs can also be divided into two parts: 
 \begin{align}\label{eq:etasplit}\eta_\vi = \begin{bmatrix} y_{1,\vi} \\ \eta_{2,\vi}\end{bmatrix}, \: \rm{where} \quad  \eta_{2,\vi}=\begin{bmatrix} y_{2,\vi}\\ \dot y_{2,\vi}\end{bmatrix}.\end{align}

We can also define switching functions, $\ResetMap_\vi$ (not $\ResetMap_{(\sis,\dos)}$ or $\ResetMap_{(\dos,\sis)}$), for the transformed statespace (from $x$ to $(\eta,z)$). For example, $\ResetMap_\dos (\eta_\dos,z_\dos) := \Phi_\sis (\ResetMap_{(\dos,\sis)} (\Phi_\dos^{-1} (\eta_\dos,z_\dos) ))$, which can in turn be split into two components  $\ResetMap^{\eta_2}_\dos$, $\ResetMap^{\PZ}_\dos $ corresponding to the coordinates $\eta_{2,\vi}$ and $(y_{1,\vi},z_\vi)$ respectively.
With this new notation, we can reformulate the hybrid invariance conditions \eqref{eq:resetmapvW} to the following:
\begin{align}\label{eq:hybridinvariance}
 \ResetMap^{\eta_2}_{\dos} (y_{1,\dos},0,z_\dos) =  0,  \quad  \ResetMap^{\eta_2}_{\sis}  (y_{1,\sis},0,z_\sis) =  0.
\end{align}

\gap
\begin{remark}
 It is important to note that the invariance conditions \eqref{eq:hybridinvariance} can be ensured only if the model is known. Therefore, this gap is addressed by including the impact based uncertainty in the following manner:
 \begin{align}\label{eq:impactuncertainty}
  \ResetMap_\sis (\eta_\sis,z_\sis) & = \hat \ResetMap_\sis (\eta_\sis,z_\sis) + \ResetMap_\sis (\eta_\sis,z_\sis) - \hat \ResetMap_\sis (\eta_\sis,z_\sis), \nonumber \\
  \Rightarrow  |\ResetMap_\sis (\eta_\sis,z_\sis)|& \leq |\ResetMap_\sis (\eta_\sis,z_\sis)| + |d_{(\sis,\dos)}|
 \end{align}
where we have denoted the impact map of the actual model of DURUS as $\ResetMap_\sis$ (for single support phase), and that of the simulated model as $\hat \ResetMap_\sis$ respectively. The difference between these post-impact maps yields the new disturbance input $d_{(\sis,\dos)} : = \ResetMap_\sis  - \hat \ResetMap_\sis $ for the discrete dynamics. Therefore, by assuming that $d_{(\sis,\dos)}$ is bounded\footnote{Boundedness of $d_\ei$ is true for a wide variety of uncertainties in robotic systems. For example, \cite{kolathaya2016parameter} showed how $d_\ei$ is bounded for parameter uncertainties for the bipedal robot AMBER.}, we can establish ISS for the walking model of DURUS. We have the following assumption:
\gap
\begin{assumption}
\label{assumption:impact}
The disturbance effects of the impact map can be minimized by identifying the system model parameters. In other words, for every $\delta >0$, there exist system model parameters that satisfy $|d_\ei|\leq \delta$ for a compact set of values of $\eta_\vi$, $z_\vi$, where $\vi$ is the source vertex of $\ei$.
\end{assumption}
\end{remark}
\gap
It is important to note that Assumption \ref{assumption:impact} does not restrict our analysis of DURUS. It if sufficient to identify basic inertial parameters of the robot like masses of each link, battery model, gear inertia and spring stiffnesses. Friction effects and models of loose parts were not included.

In the next subsection we study dynamics of the outputs where time based reference trajectories are used.

\begin{figure}[t!]
\centering
	 \includegraphics[height= 2.9cm]{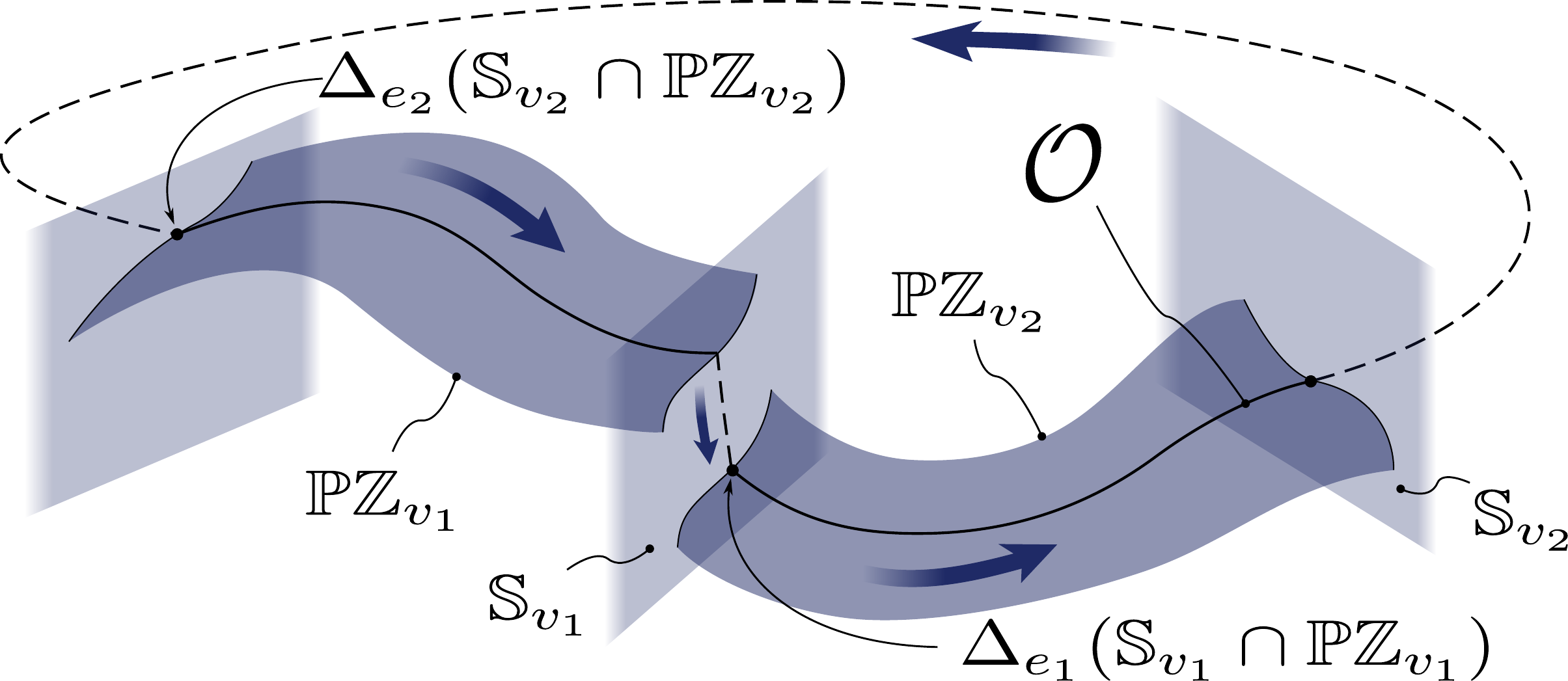}
	\caption{Figure showing a typical periodic orbit ($\mathcal{O}$) on the two dimensional partial hybrid zero dynamics (PHZD). }
	\label{fig:pzd}
\end{figure}



\subsection{Time based output dynamics}  
\label{subsec:timebasedoutputdynamics}
We studied time based tracking of joint angles in Appendix \ref{app:eisspdcontrol} for continuous robotic systems. Here, the goal is to realize time based outputs for walking. If the desired outputs are parameterized by time, we have the following output representation
\begin{align}
\label{eq:d2ot}
y^t_{2,\vi} (q) = y^a_{2,\vi}(q) - y^d_{2,\vi}(\tau_\vi(t),\alpha_\vi),
\end{align}
for the relative degree two (pose) outputs. The outputs are derived from \eqref{eq:d2o} where the phase is now dependent on time $\tau(t)$. 
The resulting output dynamics are obtained by taking the derivative
\begin{align}
\ddot{y}^t_{2} (q,\dq) &=  L_f^2 y^a_{2}(q) + L_g L_f y^a_{2}(q)u - \ddot{y}^d_{2}(\tau(t),\dot \tau(t), \ddot \tau(t), \alpha), \nonumber 
\end{align}
where the subscript $\vi$ is omitted for simpler representation. In order to drive these time dependent outputs $$\eta_t (x) = \begin{bmatrix} y_{1} (x) \\ \eta_{2,t} (x) \end{bmatrix} = \begin{bmatrix} y_1 (q,\dq) \\  y^t_2 (\q) \\ \dot{y}^t_2 (q,\dq) \end{bmatrix} \to 0,$$ we can choose time based feedback linearization:
\begin{align}
\label{eq:fblinmt}
u^t_{\rm{IO}} = \begin{bmatrix}L_{g} y_{1} \\  L_{g} L_{f} y^a_{2}\end{bmatrix} ^{-1} \left ( - \begin{bmatrix} L_{f} y_{1} \\ L_{f}^2 y^a_{2} - \ddot y^d_{2}\end{bmatrix} - \begin{bmatrix}  \epsilon y_{1} \\ 2\epsilon L_fy_{2} + \epsilon^2 y_{2} \end{bmatrix} \right ),
\end{align}
that results in the outputs going to zero exponentially.  

The time based output dynamics can be written in normal form as
\begin{align}\label{eq:dynamicsetat}
\dot{\eta}_t &=  \begin{bmatrix} L_f y^a_{1} \\ L_f y^a_{2} \\
                               L_f^2 y^a_{2} \end{bmatrix}    + \begin{bmatrix} L_g y^a_{1}   \\ 0 \\
                                                                              L_g L_f y^a_{2} \end{bmatrix} u - \begin{bmatrix} 0 \\                       	\dot{y}^d_{2} \\        \ddot{y}^d_{2} \end{bmatrix}, \quad
                                                                              &\dot{z}_t = \Psi_t(\eta_t,z_t).
\end{align}
$z_t$ are the set of zero dynamic coordinates normal to $\eta_t$ and has the invariant dynamics $\dot{z}_t = \Psi_t(0,z_t)$. 
For the time based states, $\eta_t,z_t$, we have the diffeomorphism: $\Phi_t(x)=(\eta_t(x),z_t(x))$.

\subsection{PHZD reconstruction and linear feedback laws}
\label{sec:PHZDrecons}
One of the main advantages of studying PHZD is the simpler form of the dynamics that it takes in reduced dimensions \cite{ames2014human}.
Given a suitable feedback control law, and stable PHZD, we can explicitly reconstruct the solution of the full order system from this reduced dimensional states.
We have the following linear state feedback law:
\begin{align}
 x_d &= \begin{bmatrix} q_{d,\vi} \\ \dot q_{d,\vi} \end{bmatrix} = \Phi^{-1}_\vi (y_{1,\vi},0,z_\vi) \nonumber \\
 u_{\rm{PD}}(q,\dq) &= -K_{p,\vi} (q - q_{d,\vi} ) - K_{d,\vi} (\dq - \dq_{d,\vi}), \quad \vi\in\Vertex, 
\label{eq:linearfeedbacklaw}
\end{align}
where the desired angles and velocities are obtained from the PHZD reconstruction via the inverse diffeomorphism (see \cite{ames2014human}). 
For cases where the state based parameterization $\tau(\q)$ is noisy due to poor sensing, we realize a time based parameterization of the desired angles and velocities:

\noindent\hrulefill
\begin{align}
 x^t_d &= \begin{bmatrix} q^t_{d,\vi} \\ \dot q^t_{d,\vi} \end{bmatrix} = \Phi^{-1}_{t,\vi} (y^t_{1,\vi},0,z^t_\vi) \nonumber \\
 u^t_{\rm{PD}}(q,\dq) &= -K^t_{p,\vi} (q - q^t_{d,\vi} ) - K^t_{d,\vi} (\dq - \dq^t_{d,\vi}), \quad \vi\in\Vertex, 
\label{eq:linearfeedbacklawt}
\end{align}
\hrulefill

\noindent where the time based PHZD reconstruction is utilized. Note that when $m < n$ (underactuation), the gain matrices $K_{p,\vi},K_{d,\vi},K^t_{p,\vi},K^t_{d,\vi}$ are no longer square.

For the bipedal robot, DURUS, the gain matrices are manually tuned via simulation and ensure that the input $d$ and outputs $\eta_\vi$ are ``small'' over a large number of steps (typically $50-100$). This process is repeated experimentally until the tracking errors are minimized to acceptable levels. There are, indeed, more formal ways to obtain optimal values for $K^t_{p,\vi},K^t_{d,\vi}$ such that the resulting disturbance is bounded and minimal.  
This will be studied in detail in future work.


\section{Stability of Walking }
\label{sec:theory}

In this section, we will investigate the stability of walking for both the state based controllers \eqref{eq:linearfeedbacklaw} and time based controllers \eqref{eq:linearfeedbacklawt}. It is important to note that ISS of linear feedback laws for bipedal walking robots has not been shown before. Hence, the results in this section extend the ISS results of continuous robotic systems to hybrid robotic systems. Formally, by viewing walking as an alternating sequence of double support and single support phases, we check for conditions that result in attractive and forward invariant periodic orbits. Further, we analyze the robustness by modeling the uncertainties of the system. 

Substitution of the control law \eqref{eq:linearfeedbacklaw} in \eqref{eq:dynamicseta}, and adding and subtracting \eqref{eq:fblinm} results in
\begin{align}
\label{eq:dynamicsetareform}
 \dot{\eta}_\vi  &= \begin{bmatrix}
                     L_{f_\vi} y_{1,\vi}    \\ L_{f_\vi} y_{2,\vi} \\
                     L_{f_\vi}^2 y_{2,\vi} 
                    \end{bmatrix} + \begin{bmatrix}
				    L_{g_\vi} y_{1,\vi} \\ 0\\
				    L_{g_\vi} L_{f_\vi} y_{2,\vi}
				    \end{bmatrix} u_{\rm{IO}} +  \begin{bmatrix}
				    L_{g_\vi} y_{1,\vi} \\ 0\\
				    L_{g_\vi} L_{f_\vi} y_{2,\vi}
				    \end{bmatrix} \underbrace{ u_{\rm{PD}} - u_{\rm{IO}}}_{=:d_1} \nonumber \\
 \dot{z}_\vi    &= \Psi_\vi(\eta_\vi,z_\vi), \:\: \vi \in \Vertex,
\end{align}
where the disturbance input $d_1$ is the difference between the applied control law, \eqref{eq:linearfeedbacklaw}, and the stabilizing control law, \eqref{eq:fblinm}. This formulation is similar to \eqref{eq:dynamicsreference}, wherein a standard stabilizing controller is added and subtracted to study ISS properties of the closed loop dynamics. By substituting a time based tracking control law \eqref{eq:linearfeedbacklawt} in \eqref{eq:dynamicseta} we have 
\begin{align}
\label{eq:dynamicsetareformt}
 \dot{\eta}_\vi  &= \! \begin{bmatrix}
                     L_{f_\vi} y_{1,\vi}    \\ L_{f_\vi} y_{2,\vi} \\
                     L_{f_\vi}^2 y_{2,\vi} 
                    \end{bmatrix} \! + \! \begin{bmatrix}
				    L_{g_\vi} y_{1,\vi} \\ 0\\
				    L_{g_\vi} L_{f_\vi} y_{2,\vi}
				    \end{bmatrix} u^t_{\rm{IO}}  \! + \! \begin{bmatrix}
				    L_{g_\vi} y_{1,\vi} \\ 0\\
				    L_{g_\vi} L_{f_\vi} y_{2,\vi}
				    \end{bmatrix}  \! \underbrace{ u^t_{\rm{PD}} - u^t_{\rm{IO}} }_{=:d_2}  \nonumber \\
 \dot{z}_\vi    &= \Psi_\vi(\eta_\vi,z_\vi), \:\: \vi \in \Vertex,
\end{align}
where the new disturbance input $d_2$ is defined by addition and subtraction of the time based feedback linearization $u^t_{\rm{IO}}$ \eqref{eq:fblinmt}. Furthermore, by adding and subtracting the stabilizing controller $u_{\rm{IO}}$ in \eqref{eq:dynamicsetareformt}, $d_3$ can be substituted. This means that $d_3$ can be reduced as long as the time based phase variable closely matches with the state based phase variable. This mismatch between the time and state based implementation results in the uncertainty, called phase uncertainty, discussed in detail in \cite{kolathaya2016time}. Therefore, $d_2$ corresponds to the model based uncertainties and $d_3$ corresponds to the phase based uncertainty. The approach shown in this paper then yields stability of periodic orbits under a nonzero $d=d_2+d_3$. To include the domain dependency of the controller, we add the notation $\vi$ in $d_\vi = d_{2,\vi} + d_{3,\vi}$. This yields the transverse dynamics 

\noindent\hrulefill
\begin{align}
\label{eq:dynamicsetareformfullt}
 \dot{\eta}_\vi  &= \begin{bmatrix}
                     L_{f_\vi} y_{1,\vi}    \\ L_{f_\vi} y_{2,\vi} \\
                     L_{f_\vi}^2 y_{2,\vi} 
                    \end{bmatrix} + \begin{bmatrix}
				    L_{g_\vi} y_{1,\vi} \\ 0\\
				    L_{g_\vi} L_{f_\vi} y_{2,\vi}
				    \end{bmatrix} u_{\rm{IO}}  + \begin{bmatrix}
				    L_{g_\vi} y_{1,\vi} \\ 0\\
				    L_{g_\vi} L_{f_\vi} y_{2,\vi}
				    \end{bmatrix}  d_\vi \\
\label{eq:dynamicszform}
 \dot{z}_\vi    &= \Psi_\vi(\eta_\vi,z_\vi), \:\: \vi \in \Vertex.
\end{align}
\hrulefill

\noindent Note that additional uncertainties can also be similarly modeled, and will be considered in future work.


Given the feedback control law \eqref{eq:linearfeedbacklawt}, we can establish stability of transverse dynamics w.r.t. the input $d_\vi$. In other words, we can establish exponential $d_\vi$ to $\eta_\vi$ stability. 
\gap 
\begin{lemma}{\it
 \label{lm:eiss}
 The transverse dynamics of the form \eqref{eq:dynamicsetareformfullt}, where \eqref{eq:fblinm} is substituted for $u_{\rm{IO}}$, is exponential $d_\vi$ to $\eta_\vi$ stable. In addition, due to the separability of the dynamics of $y_{1,\vi},y_{2,\vi}$, the transverse dynamics also yields exponential $d_\vi$ to $\eta_{2,\vi}$ stability.}
\end{lemma}
\gap
\begin{proof}
 Substitution of \eqref{eq:fblinm} yields output dynamics of the form 
 \begin{align}\label{eq:bdescription}
\dot \eta_\vi = \underbrace{  \begin{bmatrix}
                               -\epsilon \1 & \0 & \0 \\
                               \0 &\0 & \1 \\
                               \0 & - \epsilon^2 \1 & - 2 \epsilon \1
                              \end{bmatrix}  }_{A_\vi} \eta_\vi + \underbrace{\begin{bmatrix}
				    L_{g_\vi} y_{1,\vi} \\ 0\\
				    L_{g_\vi} L_{f_\vi} y_{2,\vi}
				    \end{bmatrix}}_{B_\vi}  d_\vi,
\end{align}
where $A_\vi$ is Hurwitz (see eqn. (11) in \cite{TAC:amesCLF}). This can be viewed in terms of Lyapunov functions:
 \begin{align}
  V_\vi (\eta_\vi) : = \eta_\vi^T P_\vi \eta_\vi,  
 \end{align}
 where $P_\vi$ is the solution to the Lyapunov equation $A^T_\vi P_\vi + P_\vi A_\vi = - Q_\vi$, $Q_\vi > 0$. The rest follows from proof of Lemma \ref{lm:eissofcontinuous}.
 To establish exponential $d_\vi$ to $\eta_{2,\vi}$ stability (see \eqref{eq:etasplit} for the description of $\eta_{2,\vi}$), we separate the dynamics to yield
 \begin{align}
  \dot \eta_{2,\vi} = \underbrace{  \begin{bmatrix}
                               \0 & \1 \\
                               - \epsilon^2 \1 & - 2 \epsilon \1
                              \end{bmatrix}  }_{A_{2,\vi}} \eta_{2,\vi} + \underbrace{\begin{bmatrix}
				             0\\
				    L_{g_\vi} L_{f_\vi} y_{2,\vi}
				    \end{bmatrix}}_{B_{2,\vi}}  d_\vi,
 \end{align}
 and the rest follows \eqref{eq:lyapunovmainderivative}, \eqref{eq:lyapunovmainderivative2} with the reduced Lyapunov function $V_{\eta_{2,\vi}} : = \eta^T_{2,\vi} P_{2,\vi} \eta_{2,\vi}$, where $P_{2,\vi} > 0 $ is, in fact, a sub-block of $P_\vi$.
\end{proof}
\gap

Given a suitable tracking control law, Lemma \ref{lm:eiss} provides an explicit way of computing ultimate bounds on the outputs. 
It is important to note that $B_\vi$ is dependent on the choice of the outputs (see \eqref{eq:bdescription} for the description of $B_\vi$), and $d$ is dependent on the choice of the control law. This provides us with an elegant way to carefully choose and tune the outputs and the control gains $K^t_{p,\vi},K^t_{d,\vi}$ such that the uncertainties are overcome in an effective manner. This result is next extended to hybrid systems, specifically, hybrid periodic orbits.

\subsection{Periodic orbits and \poincare maps}
It is a well known fact that stability of \poincare maps implies the stability of hybrid periodic orbits (and vice versa); proof of which is shown in \cite{morris2009hybrid}. This result was extended for systems with disturbance inputs in \cite{veer2017poincare}, where the same relationship was established for ISS of hybrid periodic orbits. Therefore, the goal of this subsection is to define periodic orbits, stability of periodic orbits, and the corresponding stability of \poincare maps in the context of DURUS two domain walking. This result is then extended to include ISS of periodic orbits in Section \ref{sec:stabilityhzd}.

Substitution of the control law \eqref{eq:linearfeedbacklawt} results in the dynamics \eqref{eq:dynamicsetareformfullt}. Denote its flow as
 $\phi_{t,\vi}$. 
We will first study the flow for zero disturbance $d_\vi=d_\ei=0$, and then extend for a nonzero $d$.
For the resulting hybrid dynamics, we have a periodic orbit, if, for some $(\eta^*_\dos,z^*_\dos)\in \Phi_\dos (\Guard_\dos |_x)$, $(\eta^*_\sis,z^*_\sis)\in \Phi_\sis (\Guard_\sis |_x)$,
and some $T^*_\dos , T^*_\sis > 0$,
\begin{align}
  (\eta^*_\sis,z^*_\sis) &= \phi_{T^*_\sis,\sis} \circ \Delta_\dos (\eta^*_\dos,z^*_\dos) \nonumber \\ 
    (\eta^*_\dos,z^*_\dos) & = \phi_{T^*_\dos,\dos} \circ \Delta_\sis (\eta^*_\sis,z^*_\sis) ,
\end{align}
The above equality conditions only ensure that the end point of the flow in each domain is connected with the initial point in the next domain.
\figref{fig:pzd} shows an example of a hybrid periodic orbit. For $\vi = \dos$, we have the set of points
\begin{align}
 \Orbit_\dos &= \{ \phi_{t,\dos} ( \Delta_\sis (\eta^*_\sis,z^*_\sis))   \in \Phi_\dos (\Domain_\dos |_x ) |   0 \leq t < T^*_\dos  \}.
\end{align}
We can similarly obtain the set of points $\Orbit_\sis$ for the single support phase. Hence, we can define the periodic orbit to be the pair $$\Orbit := \{ \Orbit_\dos, \Orbit_\sis \},$$ which has the period $T^* = T^*_\dos + T^*_\sis$. 
Similar formulations follow for defining a periodic orbit in the PHZD  
as the pair $$\Orbit^\PZ := \{ \Orbit^\PZ_\dos, \Orbit^\PZ_\sis \},$$ where the elements are defined via the reduced order flow $\phi^\PZ_{t,\vi}$. For example
\begin{align}
 \Orbit^\PZ_{\dos} &= \{ \phi^\PZ_{t,{\dos}} ( \Delta^\PZ_{\dos} (y^*_{1,\sis},z^*_{\sis}))   \in \Phi^\PZ_{\dos} (\PZD{\dos} ) |   0 \leq t < T^*_{\dos}  \}.\nonumber
\end{align}
Note that, $T^*_\dos,T^*_\sis$ are called the times to impact (time to reach the guard) for the corresponding flows in the domain. 
This can be generalized further to define time to impact functions for states starting from the neighborhood of the orbit. For example, for $(\eta_\sis,z_\sis) \in \B_r(\eta^*_\sis,z^*_\sis) \cap \Phi_{\sis} (\Guard_\sis|_x)$, we have
\begin{align}\label{eq:timetoimpactphase}
T_{\dos}(\eta_{\sis},z_{\sis}) = \inf \{ t\geq 0 | & \phi_{t,\dos}\circ \Delta_\sis (\eta_{\sis},z_{\sis}) \in \dots \nonumber \\
&\B_r(\eta^*_\sis,z^*_\sis) \cap \Phi_{\sis} (\Guard_\sis|_x) \}.  
\end{align}

Denote $\B_* := \B_r(\eta^*_\sis,z^*_\sis) \cap \Phi_{\sis} (\Guard_\sis|_x)$ as the neighborhood of $(\eta^*_\sis,z^*_\sis)$ intersected with the guard. Given the flows $\phi_{t,\dos},\phi_{t,\sis}$ and given the time to impact functions $T_\dos,T_\sis$, we can define \poincare maps for the initial state 
$(\eta_\sis,z_\sis) \in \B_*$ to be
\begin{align}
\label{eq:poincaremapphase}
\pcare(\eta_{\sis}, z_{{\sis}}) =  \phi_{T_{\sis},\sis} \circ \ResetMap_{\dos}  \circ \phi_{T_{\dos},\dos} \circ \ResetMap_{\sis} (\eta_{\sis},z_{\sis}).
\end{align}
The \poincare maps are mapped to and from the guard of the final domain subscript $\sis$. 
The \poincare map $\pcare$ can also be separated into two components $\pcare_{\eta_2},\pcare_{\PZ}$ corresponding to the coordinates $\eta_{2,\vi}$ and $(y_{1,\vi},z_\vi)$ respectively.

\newsec{Stability of periodic orbits.} Stability of periodic orbits can be defined via \poincare maps \cite{MOGR05}. Hence, if the \poincare map is applied $i$ times on the initial condition $(\eta^*_\sis,z^*_\sis)$, then we have the final state as $\pcare^i (\eta^*_\sis,z^*_\sis)$. Hence, we will define exponential stability for the discrete time system induced by the \poincare map.
We say that the periodic orbit $\Orbit$ is exponentially stable
if there exists an $\xi_p \in(0,1)$, $N_p>0$ such that for any initial condition $(\eta_\sis,z_\sis) \in \B_*$, 
the resulting discrete system satisfies 
\begin{align} 
| \pcare^i(\eta_\sis,z_\sis)- (\eta^*_\sis,z^*_\sis) | \leq N_p \xi^i_p | (\eta_\sis,z_\sis)- (\eta^*_\sis,z^*_\sis)|.
\end{align}
Stability of $\Orbit^\PZ$ can also be similarly defined. Also note that $\eta^*_{2,\vi} = 0$ for every $\vi$. We will discuss e-ISS of $\Orbit$ next.

\section{ISS of Hybrid Periodic Orbits}
\label{sec:stabilityhzd}
The goal of this section is to establish e-ISS of $\Orbit$ for inputs of the form \eqref{eq:linearfeedbacklawt}. In other words, the goal is to establish e-ISS of $\Orbit$ given that the reduced periodic orbit $\Orbit^\PZ$ is e-ISSable.
We will start with the definition of e-ISS for $\Orbit$, which will be again defined w.r.t. \poincare maps. Without loss of generality, we will drop the domain subscript notation for the initial states $(\eta,z) = (\eta_\sis,z_\sis)$ on the guard $\Guard|_x = \Guard_\sis|_x$, the surface $\PZD{} = \PZD{\sis}$, and also assume at $(\eta^*_\sis,z^*_\sis) = (0,0)$. Given that the disturbance $d_\vi$ is applied in addition to the control law \eqref{eq:linearfeedbacklawt}, the resulting flows $\phi_{t,\vi},\phi^\PZ_{t,\vi}$, and time to impact functions $T_\vi$ are now dependent on the disturbance $d_\vi$ in each continuous domain. Accordingly, the \poincare map $\pcare$ is now dependent on $d_\vi$, $d_\ei$, the disturbance effects from both the continuous and discrete events.
\gap
\begin{definition}{\it
\label{def:orbiteissdefinition}
The periodic orbit $\Orbit$ is e-ISS (exponential input to state stable) if there is $\xi_p \in(0,1)$, $N_p>0$ and $\iota_p \in \classK_\infty$ such that for any initial condition $(\eta,z) \in \B_*$, the resulting \poincare map satisfies
\begin{align}
| \pcare^i(\eta,z) | \leq N_p \xi^i_p | (\eta,z)| + \iota_p (\|d\|_{\max}). 
\end{align}}
\end{definition}
\gap
e-ISS of $\Orbit^\PZ$ is also similarly defined. Note that the disturbance input $\|d\|_{\max}$ is nothing but the maximum of the input disturbances in each discrete and continuous event: 
\begin{align}
\|d\|_{\max} = \max \left \{ \max_{\vi \in \Vertex} \|d_\vi\|_\infty, \max_{\ei \in \Edge} |d_\ei | \right \}.
\end{align}
Given Definition \ref{def:orbiteissdefinition}, we can now state the main theorem that establishes ISS of $\Orbit$\footnote{Proof of this main theorem is inspired by \cite{TAC:amesCLF}.}. Note that all the assumptions from \ref{assumption:1} to \ref{assumption:impact} are assumed to be valid.

\subsection{Main theorem}
\gap
\begin{theorem}{\it
\label{thm:mtiss}
Let $\Orbit^\PZ$ be an exponentially stable periodic orbit of the partial hybrid zero dynamics for a zero disturbance $\|d\|_{\max}=0$. For a sufficiently large enough $\epsilon$ there exists $\delta> 0$ such that for $\|d\|_{\max} \leq \delta$, for all initial conditions $(\eta,z)\in \B_*$, and for a linear feedback law \eqref{eq:linearfeedbacklawt}, the full order periodic orbit $\Orbit$ is e-ISS.}
\end{theorem}
\gap
Before proving Theorem \ref{thm:mtiss}, we will establish some properties of $\Orbit^\PZ$. Denote $\zeta: = (y_1,z)$, and denote the neighborhood of $\zeta_0:=(0,0)$ (which is the fixed point of $\Orbit^\PZ$ on the guard) as 
\begin{align}
 \label{eq:neighborhooddefinition}
 \B_\zeta := \mathbb{B}_r(0,0) \cap \Phi^\PZ (\PZD{} \cap \Guard |_x),
\end{align}
where the domain subscripts are suppressed for ease of notations. e-ISS of $\Orbit^\PZ$ implies that for $\|d\|_{\max}=0$ there exists $ r > 0$ such that  the following mapping 
$$\vartheta : \B_\zeta \to \B_\zeta, $$ is well defined for all $ \zeta \in \B_\zeta$. Here, $\vartheta$ is called the restricted \poincare map for the PHZD. We denote the solution to this \poincare map as $\zeta(i)$, where $i=\{0,1,2,\dots\}$, such that $\zeta({i+1}) = \vartheta(\zeta(i))$ for all $i$. \poincare map is exponentially stable if $$ |\zeta(i) | \leq  N \xi^i |\zeta(0) |,$$ for some $ N > 0$, $\xi \in (0,1)$ and all $i \geq 0$. Therefore, by converse Lyapunov theorem, there exists a Lyapunov function $V_\vartheta$, defined on $\B_\zeta $ for some $r > 0$ (possibly smaller than the previously defined $r$), and positive constants $b_1, b_2, b_3, b_4$ such that
  \begin{align}
   \label{eq:discreteVrhop1}
  &  b_1 |\zeta|^2 \leq  V_\vartheta(\zeta)  \leq  b_2 |\zeta |^2   \\
  \label{eq:discreteVrhop2}
   & V_\vartheta (\vartheta(\zeta)) - V_\vartheta(\zeta)  \leq  -b_3 |\zeta|^2 \\
   \label{eq:discreteVrhop3}
    &| V_\vartheta (\zeta) - V_\vartheta (\zeta') |  \leq  b_4 | \zeta- \zeta' |.(|\zeta| + |\zeta'|).
  \end{align}
  
Similar to \eqref{eq:timetoimpactphase}, we can also define time to impact functions for the PHZD. Denote them as $T_{\vartheta_\dos}, T_{\vartheta_\sis}$ respectively. Denote the total time to impact function on PHZD as $T_{\vartheta}:=T_{\vartheta_\dos} + T_{\vartheta_\sis}$. Similarly denote $T:=T_\dos + T_\sis$, which is obtained iteratively as
\begin{align}
 T (\eta,z) = T_\dos(\eta,z) + T_\sis (\eta_\dos(T_\dos(\eta,z)),z_\dos(T_\dos(\eta,z)))
\end{align}
The time to impact functions have upper and lower bounds. Note the disturbances $d_\vi$,$d_\ei$ were suppressed in the arguments above for ease of notations. Therefore we can obtain constants $\ubar{c}_{t} , \bar c_{t} > 0$ such that
\begin{align}
 \ubar{c}_{t} T^*_{\vi} \leq  T_{\vartheta_\vi}  \leq  \bar c_{t} T^*_{\vi}  \quad \quad \ubar{c}_{t} T^*_{\vi}  \leq  T_{\vi} \leq  \bar c_{t} T^*_{\vi}.
\end{align}
Note that $T^*_\vi$ are the times to impact on the periodic orbit $\Orbit$.
We will state the following Lemma that is required to prove Theorem \ref{thm:mtiss}.

\gap
\begin{lemma}{\it
  \label{lm:mainlemmaphase}
  Let $\Orbit^{\PZ}$ be an exponentially stable periodic orbit of the hybrid zero dynamics under a zero disturbance $\|d\|_{\max}=0$. Given the linear feedback law \eqref{eq:linearfeedbacklawt} that renders the transverse dynamics \eqref{eq:dynamicsetareformfullt} e-ISS in the continuous dynamics, there exist constants $A_1$,$A_2$,$D_1$,$D_2  > 0$ such that for all $(\eta,z) \in \B_* $
  \begin{align}
   \label{eq:firstlemmaphase} | T(\eta,z) - T_\vartheta(y_1,z) | &\leq  A_1 |\eta_2|  + D_1 \|d\|_{\max}  \\
   \label{eq:secondlemmaphase} | \pcare_\PZ (\eta,z) - \vartheta(y_1,z) | & \leq A_2 | \eta_2| + D_2 \|d\|_{\max}
  \end{align}}
  \end{lemma}\gap
  \begin{proof}  
  By a slight abuse of notations, we start with the initial condition on the single support phase $(\eta_2,y_1,z) = (\eta_2,\zeta) \in \B_*$, and then observe the time solution $(\eta_{2,\vi}(t),\zeta_\vi(t))$ in each domain for continuous dynamics of the form \eqref{eq:dynamicsetareformfullt}. The goal is to compare the evolution of the resulting trajectory undergoing disturbance with the trajectory of the orbit on PHZD undergoing no external input disturbance. Therefore, for ease of understanding, we will denote the actual trajectory in each domain as $(\eta_{2,\vi}(t),\zeta^a_\vi(t))$ and the base trajectory (of the PHZD) as $(0,\zeta^b_\vi(t))$. With these two trajectories an auxiliary time to impact function $T_B$ was defined in \cite[eqn. (55)]{TAC:amesCLF}, which is reformulated w.r.t. each $\vi\in\Vertex$ here as
  \begin{align}
  & T_{B,\vi}(\mu_1,\mu_2,\zeta_\vi) = \inf \{ t \geq 0 : h_\vi(\mu_1,\zeta^b_\vi(t) + \mu_2) = 0 \} \nonumber 
  \end{align}
  where $h_\vi$ is a smooth function that indicates the guard ``strike'' condition. For the double support phase $h_\dos$ is the vertical ground reaction force, and for the single support phase $h_\sis$ is the height of the swinging foot from ground. Transition to the next domain is triggered when the guard condition $h$ crosses zero. If $\mu_1$,$\mu_2$ are defined as
  \begin{align}
  \mu_1 = \eta_{2,\vi}(t) |_{t = T_\vi}, \quad \mu_2 =  \zeta^a_\vi(t) - \zeta^b_\vi(t) |_{t = T_\vi},
  \end{align}
  then it can be observed that $T_{B,\vi} = T_\vi$. Therefore by the property of Lipschitz continuity of $T_{B,\vi}$, we have
  \begin{align}
  \label{eq:auxiliarytimetoimpact}
  &  | T_{\vi}  - T_{\vartheta_\vi} | \leq L_{B,\vi} ( |\eta_{2,\vi}(T_\vi)| + |\zeta^a_\vi(T_\vi) - \zeta^b_\vi(T_\vi)| ),
  \end{align}
where $L_{B,\vi}$ is the Lipschitz constant. It can be observed that \eqref{eq:auxiliarytimetoimpact} is not dependent on $\epsilon$. In order to find the total time difference \eqref{eq:firstlemmaphase}, the goal is to obtain the norms on the RHS of \eqref{eq:auxiliarytimetoimpact} for each $\vi$ and summing the resulting two inequalities. 
  
  \newsec{First norm in RHS of \eqref{eq:auxiliarytimetoimpact}.} For each $\vi\in\Vertex$, there exist $c_{1,\vi}$, $c_{2,\vi}$, $c_{3,\vi} > 0$, and $D_{1,\vi}>0$ such that \cite[eqn. (21)]{TAC:amesCLF}
 \begin{align}
   | \eta_{2,\vi}( T_{\vi} ) | &= | \eta_{2,\vi}(t) |_{t= T_{\vi} } \nonumber \\
   & \leq   \sqrt{\frac{c_{2,\vi}}{c_{1,\vi}}} \epsilon e^{- \frac{c_{3,\vi}\epsilon}{2 }  T_{\vi}}   |\eta_{2,\vi}(0)| + D_{1,\vi} \|d_\vi\|_\infty \nonumber  \\
\label{eq:intermediatevalueinitialcondition}   & \leq   {\sqrt{\frac{c_{2,\vi}}{c_{1,\vi}}} \epsilon e^{- \frac{c_{3,\vi}\epsilon}{2 } \ubar{c}_t T^*_{\vi}}}   |\eta_{2,\vi}(0)| + D_{1,\vi} \|d_\vi\|_\infty \\
  \label{eq:finalvalueinitlacondiion}   & \leq   C_{1,\vi}   |\eta_{2,\vi}(0)| + D_{1,\vi} \|d_\vi\|_\infty ,   \end{align}
   where $C_{1,\vi}$ is some constant that is not dependent on $\epsilon$ (see \cite[eqns. following (59)]{TAC:amesCLF}). $D_{1,\vi}$ can be obtained via expressions similar to \eqref{eq:lyapunovmainderivative2}, \eqref{eq:etaboundmaadi}. It is important to note that the evolution of the dynamics of $y_{1,\vi}$ can also be derived similarly and omitted for convenience.
   
   Since $(\eta_{\sis}(0), z_\sis(0) )= \ResetMap_\dos (\eta_{\dos}(T_\dos), z_\dos(T_\dos) )$, we can easily establish that 
   \begin{align}
   \label{eq:impactdist1}
   |\eta_{2,\sis}(0)| \leq L_{2,\dos} |\eta_{2,\dos}(T_\dos)| + |d_{(\dos,\sis)}|,
   \end{align}
  where the disturbance $d_{(\dos,\sis)}$ appears due to \eqref{eq:impactuncertainty}.
  $L_{2,\dos}$ is the Lipschitz constant of $\eta_2$ component of $\ResetMap_\dos$. Similarly  
  \begin{align}
  \label{eq:impactdist2}
 |\eta_{2,\dos}(0)| \leq L_{2,\sis} |\eta_{2}| + |d_{(\sis,\dos)}|.
 \end{align}
 
We therefore have that
 \begin{align}
      & | \eta_{2,\sis}( T_{\sis} ) |   \leq  C_{1,\sis} L_{2,\dos} C_{1,\dos} L_{2,\sis} |\eta_2| +  D_\eta \|d\|_{\max}, 
 \end{align}
 after substituting for \eqref{eq:finalvalueinitlacondiion}, \eqref{eq:impactdist1}, \eqref{eq:impactdist2} for each continuous and discrete event.
$D_\eta$ is also obtained accordingly, by grouping and replacing the individual disturbances with $\|d\|_{\max}$.


\newsec{Second norm in RHS of \eqref{eq:auxiliarytimetoimpact}.}
We have the following result (proof for a similar result was shown in \cite{kolathaya2016parameter} for parameter uncertainty). We will drop the vertices $\vi\in\Vertex$ for ease of notations.
 \begin{align}
 \label{eq:integralequatioinzeta}
   |y^a_{1}(t) - y^b_{1}(t)|  \leq  e^{-\frac{\epsilon }{2}t} \:\:\: | & y^a_{1}  (0)  - y^b_{1}(0) | + D_{y}\|d\|_\infty   \nonumber \\
  |z^a(t) - z^b(t)|  \leq  | z^a(0) & - z^b(0) | \nonumber\\ 
			     +  \int_0^t & \left ( \begin{bmatrix}
                              \Psi(y^a_1(t'),\eta_2(t'),z^a(t'))
                              \end{bmatrix} \right. \dots \nonumber \\
                             & \left. - \begin{bmatrix}
						   \Psi(y^b_1(t'),0,z^b(t'))
					      \end{bmatrix}   \right ) dt'  \\
 |\zeta^a(t) - \zeta^b(t)|  \leq  |  y^a_{1}(t) & - y^b_{1}(t)| + |z^a(t) - z^b(t)|, \nonumber
 \end{align}
 where the final values of the states are obtained as a function of the initial values in each domain. $D_y$ is some constant which is obtained due to the fact that the velocity output $y_1^a$ is e-ISS in each continuous dynamics. Note that $y^a_1$,$z^a$ are the components of $\zeta^a$ (similarly for $\zeta^b$). We also know that the initial values $\zeta^a(0)=(y^a_1(0),z^a(0))$, $\zeta^b(0)=(y^b_1(0),z^b(0)) $ depend on the states from the previous domain. Therefore 
 \begin{align}  
 \label{eq:zetainitialdifference}
|\zeta^a_\dos (0) - \zeta^b_\dos (0) | \leq & L_{1,\sis} |\eta_{2}| + |d_{(\sis,\dos)}|   \\
|\zeta^a_\sis (0) - \zeta^b_\sis (0) | \leq & L_{1,\dos} |\zeta^a_{\dos}(T_\dos)-\zeta^b_{\dos}(T_\dos)| \nonumber \\
				      & \qquad + L_{1,\dos} |\eta_{2,\dos}(T_\dos)| + |d_{(\dos,\sis)}|, \nonumber
 \end{align}
 where $L_{1,\sis}$, $L_{1,\dos}$ are the Lipschitz constants. 
 Therefore, we have the following inequality:
 \begin{align}\label{eq:zetadifferencemaadi}
| \zeta^a(t) - \zeta^b(t) |  \leq  &   |  y^a_{1}  (0)  - y^b_{1}(0) | + D_{y} \|d\|_{\max}  \nonumber \\
				     & +  | z^a(0)  - z^b(0) |    \\
				     & + \int_0^t L_q ( |\eta_2(t')| + |\zeta^a(t')  - \zeta^b(t') | ) dt', \nonumber
 \end{align}
 where $L_q$ is the Lipschitz constant of $\Psi$. In \eqref{eq:zetadifferencemaadi}, $|\zeta^a(0)  - \zeta^b(0) |$ can be replaced with \eqref{eq:zetainitialdifference} (based on the domain), and $\eta_2(t')$ can be replaced with \eqref{eq:finalvalueinitlacondiion}, and the final time $t$ can be replaced with the time to impact $T_\vi$. Initial values of $y_1$,$z$at each domain will have
 \begin{align}
  \label{eq:velocityoutputinequality}
  |y^a_{1,\vi} (0)  - y^b_{1,\vi} (0) |  &\leq |\zeta^a_\vi (0) - \zeta^b_\vi (0) | \nonumber \\
  |z^a_{\vi} (0)  - z^b_{\vi} (0) |  &\leq |\zeta^a_\vi (0) - \zeta^b_\vi (0) |,
 \end{align}
 which can be replaced by \eqref{eq:zetainitialdifference}.
 The final result will look like the following for some constants $C_{2,\vi}$, $D_{2,\vi}$ (also see \cite[eqn. (60)]{TAC:amesCLF})
 \begin{align}\label{eq:finalzetadifferencedesctiptions}
  | \zeta^a_\vi(T_\vi) - \zeta^b_\vi(T_\vi) | \leq C_{2,\vi} |\eta_2 | + D_{2,\vi} \|d\|_{\max}.
 \end{align}
The above result and \eqref{eq:finalvalueinitlacondiion} can be substituted in \eqref{eq:auxiliarytimetoimpact} to obtain the final form for \eqref{eq:firstlemmaphase}.

To prove \eqref{eq:secondlemmaphase} define
 \begin{align}\label{eq:psiupperbound}
    C_{3,\vi}  & = \max_{\ubar{c} T^*_\vi \leq T_\vi \leq \bar c T^*_\vi} \left | \Psi_\vi (0,\zeta^{b}_\vi(t)) \right| 
 \end{align}
 Since $\mathbb{P}_\PZ$ is the PHZD-component of the \poincare map $\pcare$, it then follows that
\begin{align}
 | \pcare_\PZ (\eta,z)  - \vartheta(\zeta) |& = | \zeta^a_\sis (T_\sis)  - \zeta^b_\sis (T_{\vartheta_\sis}) | \nonumber \\
					    & \leq | \zeta^a_\sis (T_\sis)  - \zeta^b_\sis (T_{\sis})|  \nonumber  \\
					    & \qquad + \left | \int_{T_{\vartheta_\sis}}^{T_\sis} \begin{bmatrix}
						  -\epsilon y^b_1(t') \\
						   \Psi(0,\zeta^b(t'))
					      \end{bmatrix} dt'\right | \nonumber \\
					      & \leq | \zeta^a_\sis (T_\sis)  - \zeta^b_\sis (T_{\sis})| + \left |\int_{T_{\vartheta_\sis}}^{T_\sis}\epsilon y^b_{1,\sis}(t') dt' \right |\nonumber \\
					      & \qquad  +  C_{3,\vi} |T_\sis - T_{\vartheta_\sis} | .
\end{align}
Since $y^b_{1,\sis}(t)$ is exponentially decaying, the term $\epsilon y^b_{1,\sis}(t)$ has an upper bound independent of $\epsilon$. Similarly, evolution of $\Psi$ on the PHZD has the upper bound \eqref{eq:psiupperbound}. Therefore, substituting \eqref{eq:finalzetadifferencedesctiptions} for the first RHS term above and using the previously proven result \eqref{eq:firstlemmaphase}, establishes \eqref{eq:secondlemmaphase}.
  \end{proof}
 \gap
We will now show the proof of Theorem \ref{thm:mtiss}.

\subsection{Proof of main theorem}

\noindent\begin{proof}[Proof of Theorem \ref{thm:mtiss}]
We start by picking a suitable value of $\epsilon$, as shown by \cite[Theorem 2]{TAC:amesCLF} that yields exponential convergence of the periodic orbit $\Orbit$ under a zero disturbance $\|d\|_{\max}=0$. In order to establish e-ISS of $\Orbit$, it is sufficient to show that the \poincare map $\pcare$ \eqref{eq:poincaremapphase} is e-ISS \cite{veer2017poincare}. Hence, the goal now is to obtain an ISS-Lyapunov function of the form \eqref{eq:ISSdstricter} for the \poincare map.

For the Re-ISS-CLF $V_\epsilon$ (domain subscript $\sis$ is suppressed), denote its reduced Lyapunov function (of only $\eta_2$ coordinates) and restriction to the switching surface by $V_{\epsilon,\eta_2}$. It can be verified that the matrix $P_\epsilon$ can be separated into two block matrices, with the latter being the matrix used to obtain the Lyapunov function $V_{\epsilon,\eta_2}$. 
With these two Lyapunov functions we define the following candidate Lyapunov function:
  \begin{align}
   V_{P} (\eta,z) = V_\vartheta (\zeta) + \sigma V_{\epsilon,\eta_2} (\eta_2)
  \end{align}
defined on $B_*$.
The lower and upper bounds on $V_P$ are
 $$\min \{ b_1,  \sigma c_1 \} |(\eta,z)|^2 , \quad\max \{ b_2,  \frac{\sigma c_2}{\epsilon^2} \} |(\eta,z)|^2 $$
respectively. $c_1 (= c_{1,\sis})$ and $\frac{c_{2}}{\epsilon^2} (= \frac{c_{2,\sis}}{\epsilon^2}) $ are the maximum and minimum eigenvalues of $P_{\epsilon}$. 
Since the \poincare map $\pcare$ can be divided into two components $\pcare_{\eta_2},\pcare_\PZ$, we have
\begin{align}
V_{\epsilon,\eta_2} & (\pcare_{\eta_2}(\eta,z)) \leq \epsilon^2 c_{2,\sis} |\eta_{2,\sis}(T_\sis)|^2, 
 \end{align}
 where \eqref{eq:intermediatevalueinitialcondition} can be substituted to yield the following inequality for some constants $A_3$,$A_4$,$D'_\eta > 0$:
  \begin{align}
 V_{\epsilon,\eta_2} & (\pcare_{\eta_2}(\eta,z))  \leq A_3|\eta_2 | ^2 + A_4 |\eta_2 | \|d\|_{\max} + D'_{\eta} \|d\|^2_{\max},\nonumber
 \end{align}
Note that $A_3$ decreases as $\epsilon$ increases\footnote{This is the idea behind the notion of {\it rapid exponential convergence} (as shown by \cite{TAC:amesCLF}), where a suitable $\epsilon$ is picked in order to ensure that $A_3$ is small enough.}.  
Hence, we have the following:
\begin{align}
V_{\epsilon,\eta_2} (\pcare_{\eta_2} & (\eta,z)) - V_{\epsilon,\eta_2} (\eta_2) \nonumber \\
 & \leq  A_3 |\eta_2 |^2 + A_4 |\eta_2 | \|d\|_{\max} + D'_{\eta} \|d\|^2_{\max} - c_1 | \eta_2 |^2 \nonumber	
\end{align}
We also have the following by using \eqref{eq:secondlemmaphase}:
\begin{eqnarray}
| \pcare_\PZ( \eta,z) | &=& |\pcare_\PZ(\eta,z) -\vartheta(\zeta) + \vartheta(\zeta) - \vartheta(0) | \nonumber \\
                      & \leq & A_2 |\eta_2 | + D_2 \|d\|_{\max} +  L_\vartheta |\zeta| ,
\end{eqnarray}
where $L_\vartheta$ is the Lipschitz constant of $\vartheta(\zeta)$. From \eqref{eq:discreteVrhop3}
\begin{align}
\label{eq:diffVrho1p}
 V_\vartheta (\pcare_\PZ( \eta,z)) - & V_\vartheta (\vartheta(\zeta))    \leq  b_4 ( A_2 |\eta_2 | + D_2 \|d\|_{\max} )   \\
& (A_2 |\eta_2 | + D_2 \|d\|_{\max} + (L_\vartheta + N \xi ) |\zeta|). \nonumber 
\end{align}
It follows that 
\begin{eqnarray}
V_\vartheta (\pcare_\PZ( \eta,z)) - V_\vartheta (\zeta) & =  &  V_\vartheta (\pcare_\PZ( \eta,z)) - V_\vartheta (\vartheta(\zeta)) \nonumber \\
										& &	+ V_\vartheta(\vartheta(\zeta)) - V_\vartheta(\zeta) ,
\end{eqnarray}
and the expressions in \eqref{eq:diffVrho1p} and in \eqref{eq:discreteVrhop2} can be substituted. Combining the entire Lyapunov function we have
\begin{align}
\label{eq:finaldiffp}
V_P ( \pcare(\eta,z) ) - V_P (\eta,z) \leq  - \left [ \begin{array}{c}
                                                         | \eta_2 | \\
                                                         |\zeta| \\
                                                         \|d \|_{\max}
                                                        \end{array} \right ]^T \Lambda_{\HybridSystem} \left [ \begin{array}{c}
                                                         | \eta_2 | \\
                                                         |\zeta|  \\
                                                         \|d \|_{\max} \end{array} \right ] \nonumber
\end{align}
where the symmetric matrix $\Lambda_{\HybridSystem} \in \R^{ 3 \times 3} $ is similar to the formulation given in \cite[eqn. $(92)$]{kolathaya2016parameter}:
\begin{align}
a_1 = \Lambda_\HybridSystem(1,1) &= \sigma(c_1 - A_3) - b_4 A_2^2 \nonumber \\
a_2 = \Lambda_\HybridSystem(1,2) &= -\frac{b_4 A_2}{2} (L_\vartheta + N \xi ) \nonumber \\
a_3 = \Lambda_\HybridSystem(1,3) &= \frac{-b_4 A_2 D_2 - \sigma A_4}{2}  \nonumber \\
a_4 = \Lambda_\HybridSystem(2,2) &= b_3 \nonumber \\
a_5 = \Lambda_\HybridSystem(2,3) &=  -\frac{b_4 D_2}{2} (L_\vartheta + N \xi ) \nonumber \\
a_6 = \Lambda_\HybridSystem(3,3) &=  -b_4 D_2^2 - \sigma D_\eta. 
\end{align}
Rest of the proof is similar to \cite[eqns. $(93)$ to $(96)$]{kolathaya2016parameter}, where the following inequality is obtained:
\begin{align}
\label{eq:ISSLDerivative}
V_P ( \pcare( & \eta,z) )  - V_P (\eta,z)  \\ 
 & \leq  - \frac{A_5}{2} |(\eta,z)|^2 + \left ( \frac{A_6^2}{A_5} + b_4 D_2^2 + \sigma D_\eta \right ) \|d\|^2_{\max}, \nonumber
\end{align}
where
\begin{align}
A_5 & = \lambda_{\min}\left ( \begin{bmatrix}
								a_1 & a_2 \\ a_2 & a_4
								\end{bmatrix} \right ) \nonumber\\
A_6 & = b_4 A_2 D_2 + \sigma A_4 + b_4 D_2 (L_\vartheta + N \xi), \nonumber
\end{align}
where the positivity of $A_5$ is ensured by picking sufficiently large enough $\sigma$.
It can be verified that \eqref{eq:ISSLDerivative} is of the form \eqref{eq:ISSdstricter}, thereby establishing e-ISS. The ultimate bound on $(\eta,z)$ can be easily obtained when \eqref{eq:ISSLDerivative} is changed to an equality. Therefore, we can pick an appropriate $\delta>0$ in order to ensure that $(\eta,z)$ is well inside $\B_*$ (for example, see \cite[eqn. (97)]{kolathaya2016parameter}).
 \end{proof}
\gap


%

\section{Results and conclusions}
\label{sec:results}

For verification of the results presented in this work, walking controllers demonstrating ISS are implemented on DURUS in both simulation and experiment. DURUS consists of fifteen actuated joints and one linear passive spring at the end of each leg. The generalized coordinates of the robot are described in \cite{kolathaya2016time} and the continuous dynamics of the bipedal robot is given by \eqref{eq:eom-general}. The nominal walking gait considered has two phases: single support, and double support, as shown in \figref{fig:hs}. A stable reference walking gait is obtained via an offline optimization algorithm \cite{Hereid_etal_2016}. 
\begin{figure}[h!]
	\includegraphics[width=0.98\columnwidth]{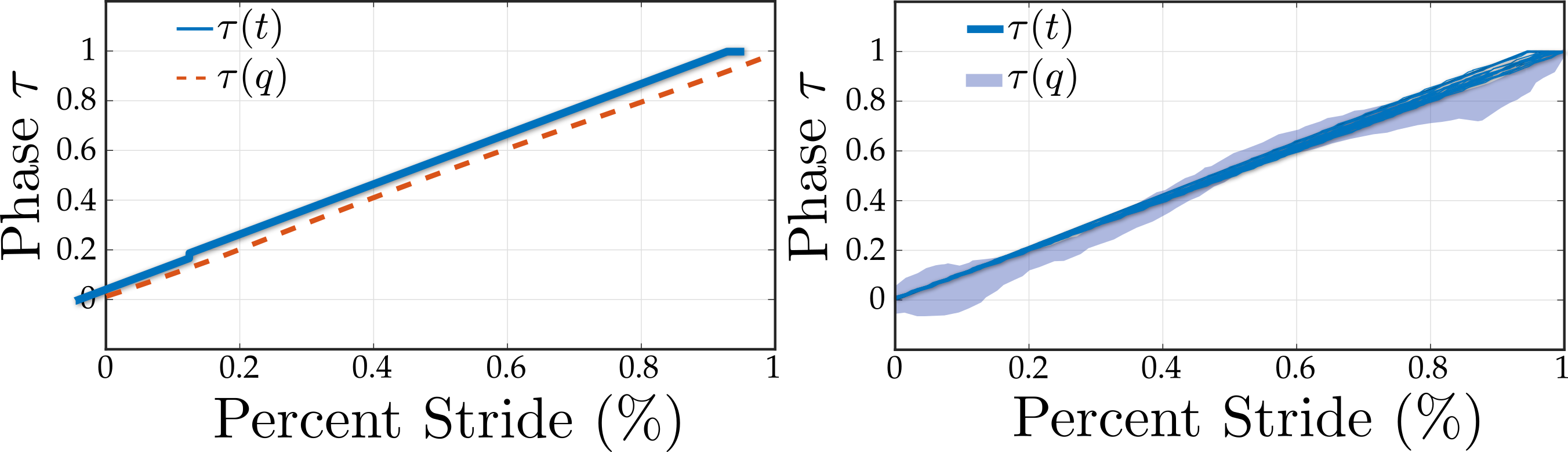} 
	\caption{Phase variable comparison between simulation (left) and experiment (right).}
	\label{fig:resultstau}
\end{figure}

\newsec{Outputs.} The subscripts $a,k,h,w$ represent ankle, knee, hip and waist respectively; while $r,p,y$ represent roll, pitch, and yaw. Therefore, $\q_{rkp}$ would indicate right knee pitch angle and $\q_{lhr}$ would indicate left hip roll angle. With these notations, we define a relative degree one output as
 $y^1(\q) = \delta \dot{p}_{hip}(\q) - v_d$,
where $\delta p_{hip}(\q)$ is the linearized hip position,
\begin{align}
  \label{eq:linearized-hip-position}
  \!\! \delta p_{hip}(\q)\!\!=\!\!l_a \q_{rap} + (l_a + l_c) \q_{rkp} + (l_a + l_c +  l_t) \q_{rhp},
\end{align}
with $l_a$, $l_c$, and $l_t$ the length of ankle, calf, and thigh link of the robot respectively. $v_d$ is a constant desired velocity. The relative degree two outputs are defined in the following (assuming left leg is the stance leg):
\begin{itemize}
  \item knee pitches: $\q_{rkp},\q_{lkp}$
  \item torso pitch: $-\q_{lap}-\q_{lkp}-\q_{lhp}$ 
  \item torso roll: $-\q_{rar}-\q_{rhr}$
  \item ankle roll:  $\q_{lar}$ 
  \item hip yaw:  $\q_{lhy}$ 
  \item waist: $ \q_{wr},\q_{wp},\q_{wy}$
  \item nonstance slope: $-\q_{lap} - \q_{lkp} -   \q_{lhp} + \frac{l_c}{l_c + l_t}\q_{rkp} + \q_{rhp}$ 
  \item leg roll: $ p^h(\q) - p^y(\q)$
  \item nonstance foot: $p^x,p^y,p^z$
\end{itemize}
where $p^h(\q)$ is the $y$ position of the right hip, and $p^x(\q),p^y(\q),p^z(\q)$ are the $x,y,z$ positions of the nonstance foot respectively. These $14$ outputs are denoted together as $y^a_2(\q)$. 
The outputs of the system are then defined as
\begin{align*}
	y(q) &= \underbrace{\begin{bmatrix} y^a_1 \\ y^a_2 \end{bmatrix}}_{y^a(q)} - \underbrace{\begin{bmatrix} v_d \\ y_2^d(\tau,\alpha) \end{bmatrix}}_{y^d(\tau, \alpha)},
\end{align*}
where the desired output functions are parameterized by the phase variable $\tau(\q)$,
defined as 
\begin{align}
 \label{eq:phasevariable}
\tau(\q):= \frac{\delta p_{hip}(\q) - \delta p_{hip}(\q^+)}{v_d},
\end{align}
\begin{figure*}[t!]
	\includegraphics[width=1\textwidth]{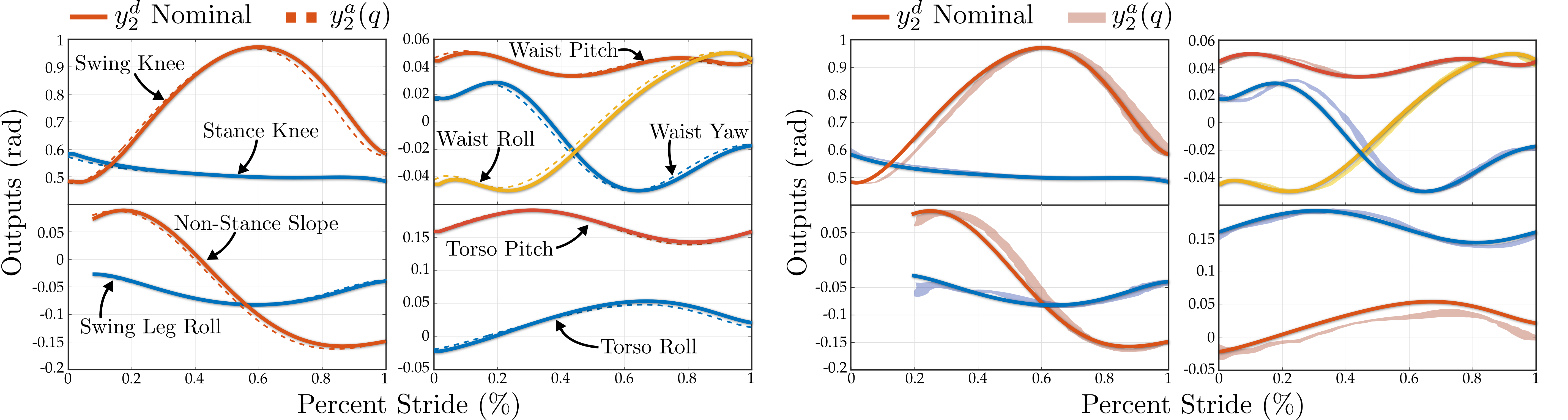} 
	\caption{Figures showing the simulation and experimental results from the beginning to end of a step. The desired outputs are obtained via B\'ezier polynomials parameterized by the linearized hip position \eqref{eq:phasevariable}.}
	\label{fig:resultssimexp}
\end{figure*}
\begin{figure*}[h!]\centering
	\includegraphics[width=0.14\textwidth]{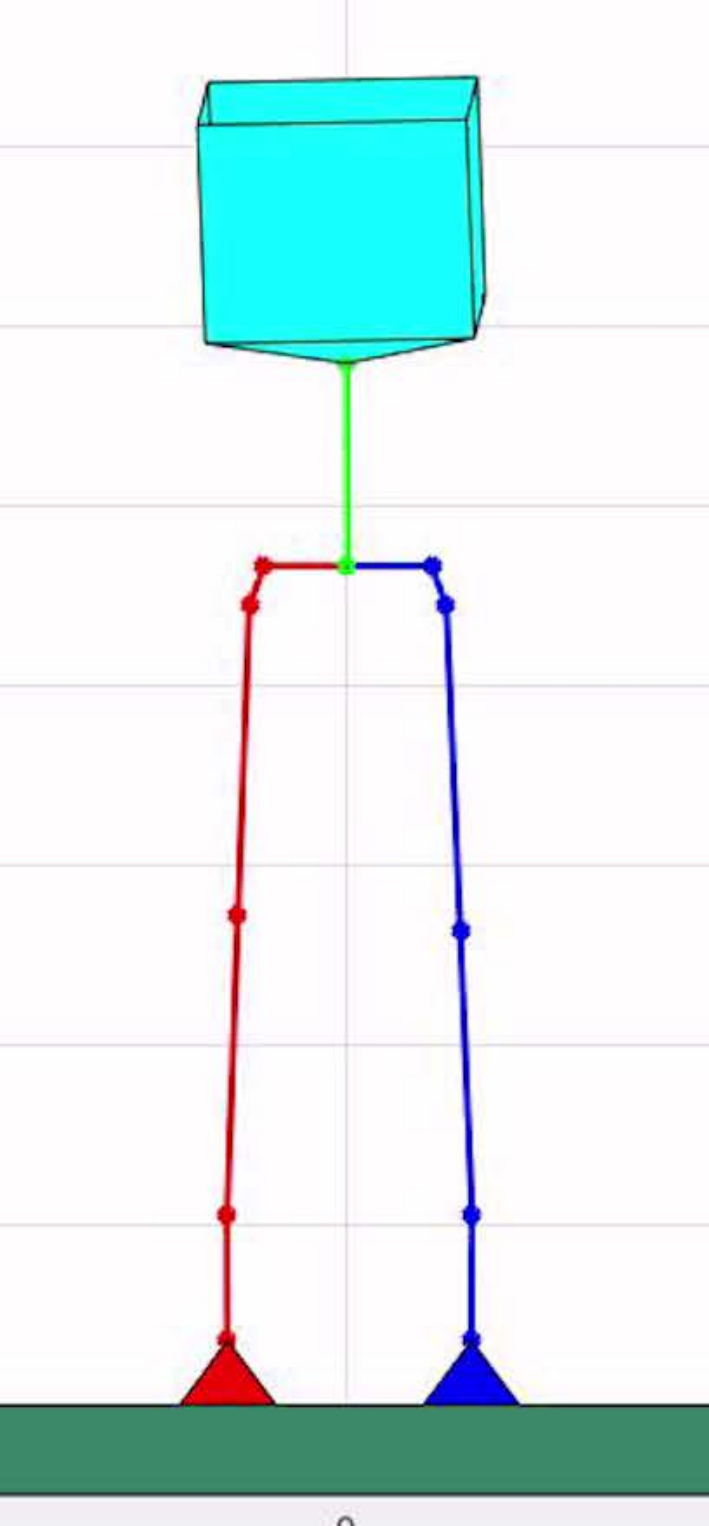}\hspace{-1mm}
	\includegraphics[width=0.14\textwidth]{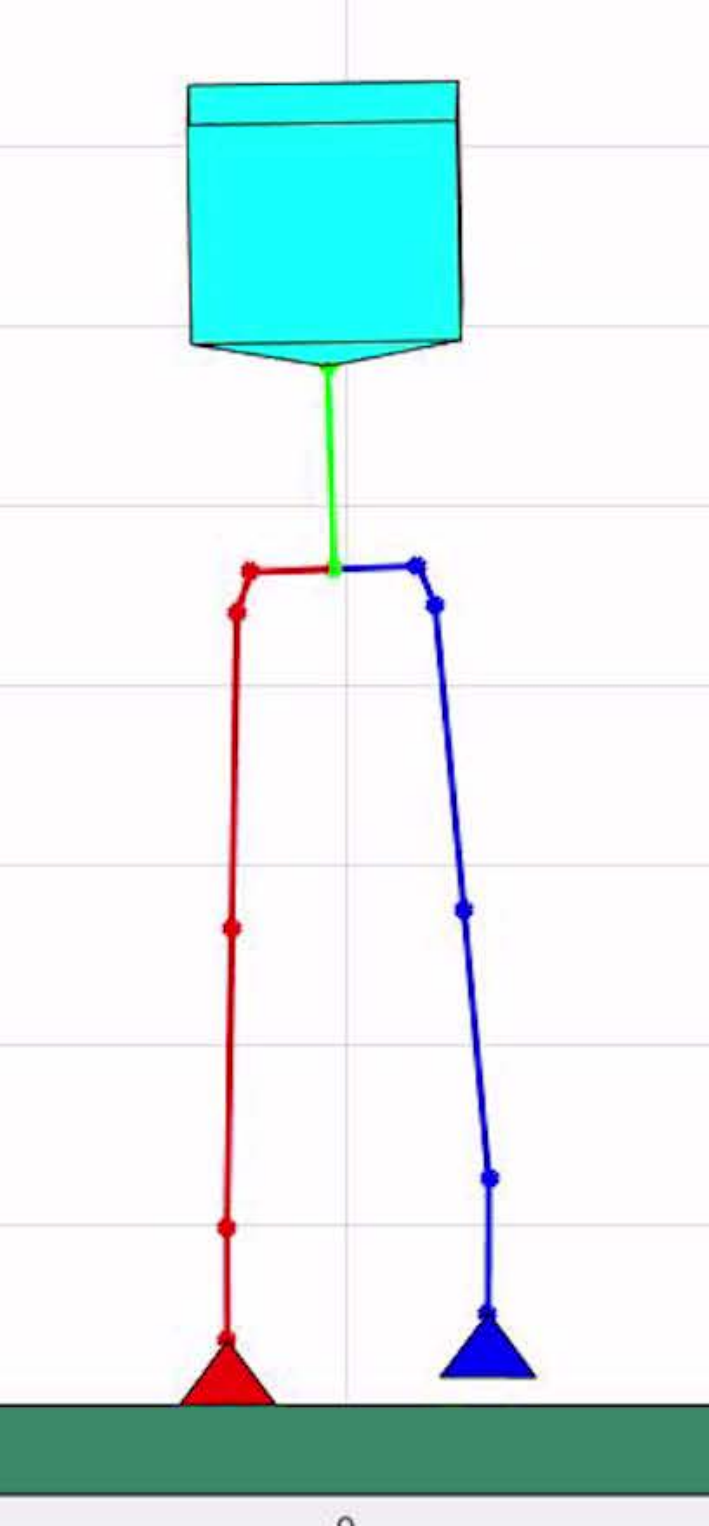}\hspace{-1mm}
	\includegraphics[width=0.14\textwidth]{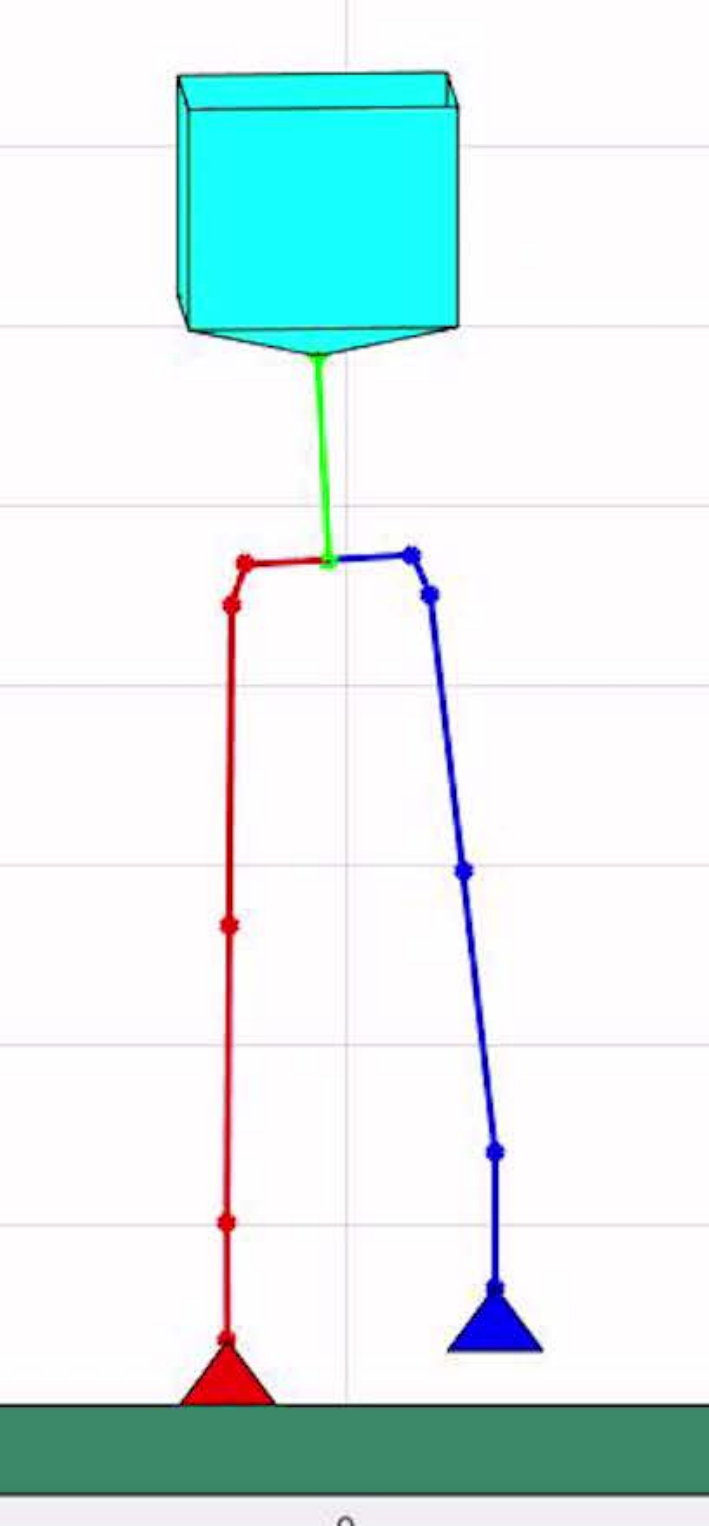}\hspace{-1mm}
	\includegraphics[width=0.14\textwidth]{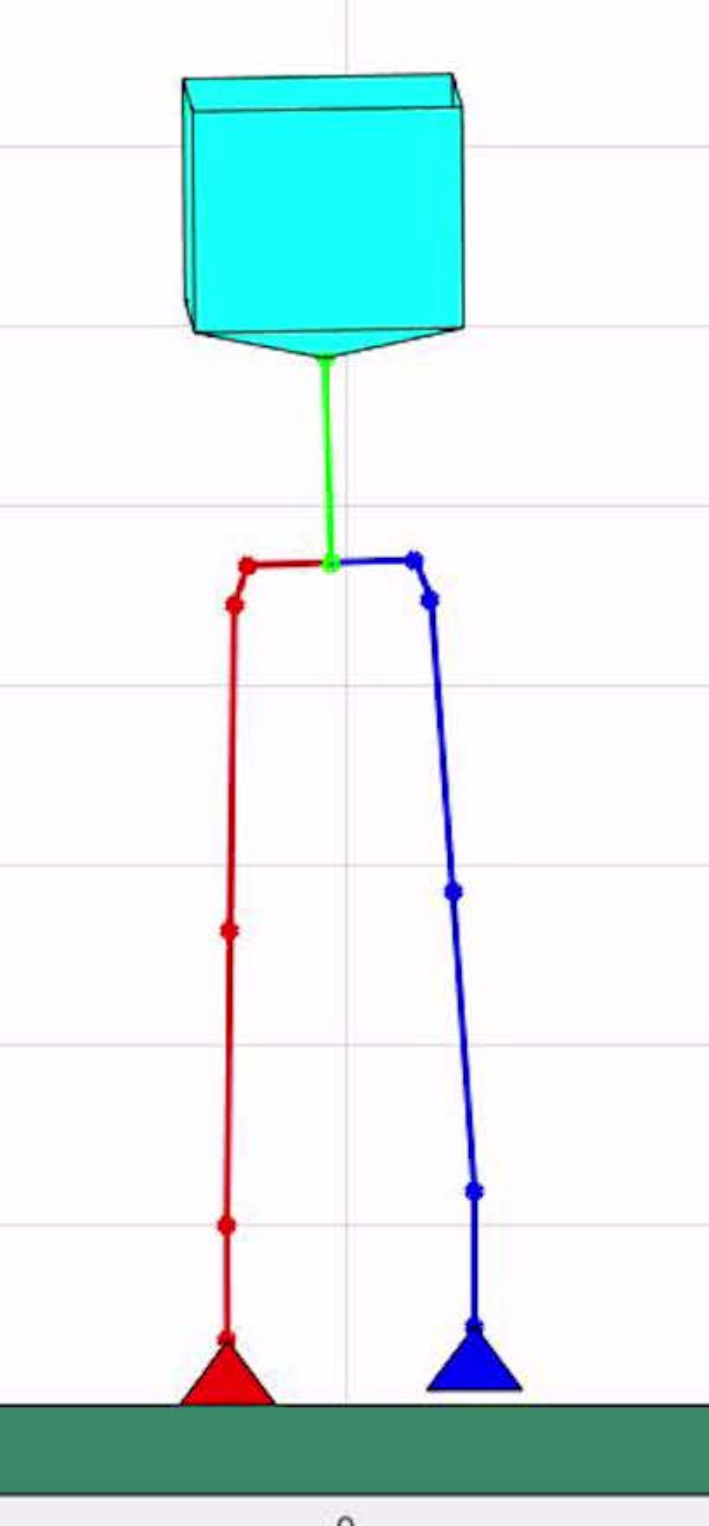}\hspace{-1mm}
	\includegraphics[width=0.14\textwidth]{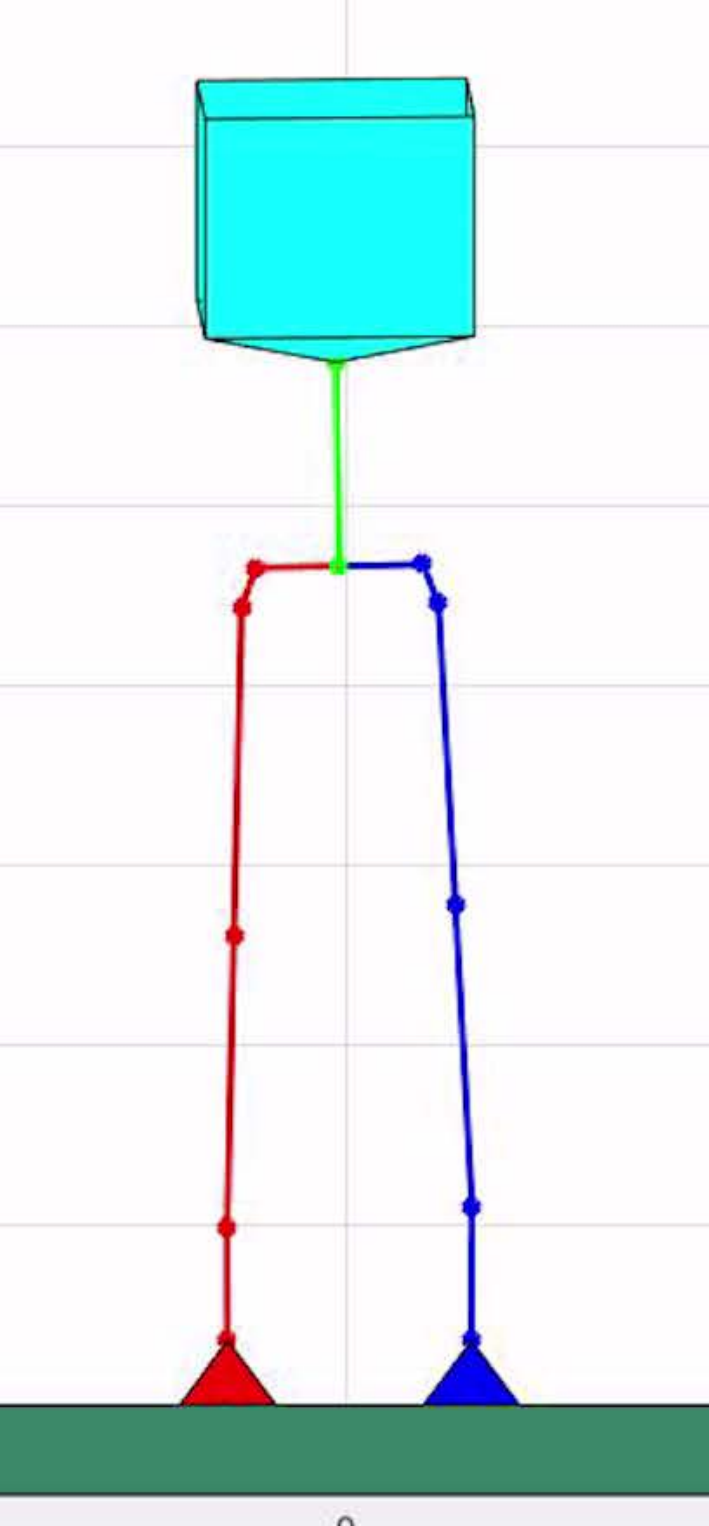}\hspace{-1mm}
	\includegraphics[width=0.14\textwidth]{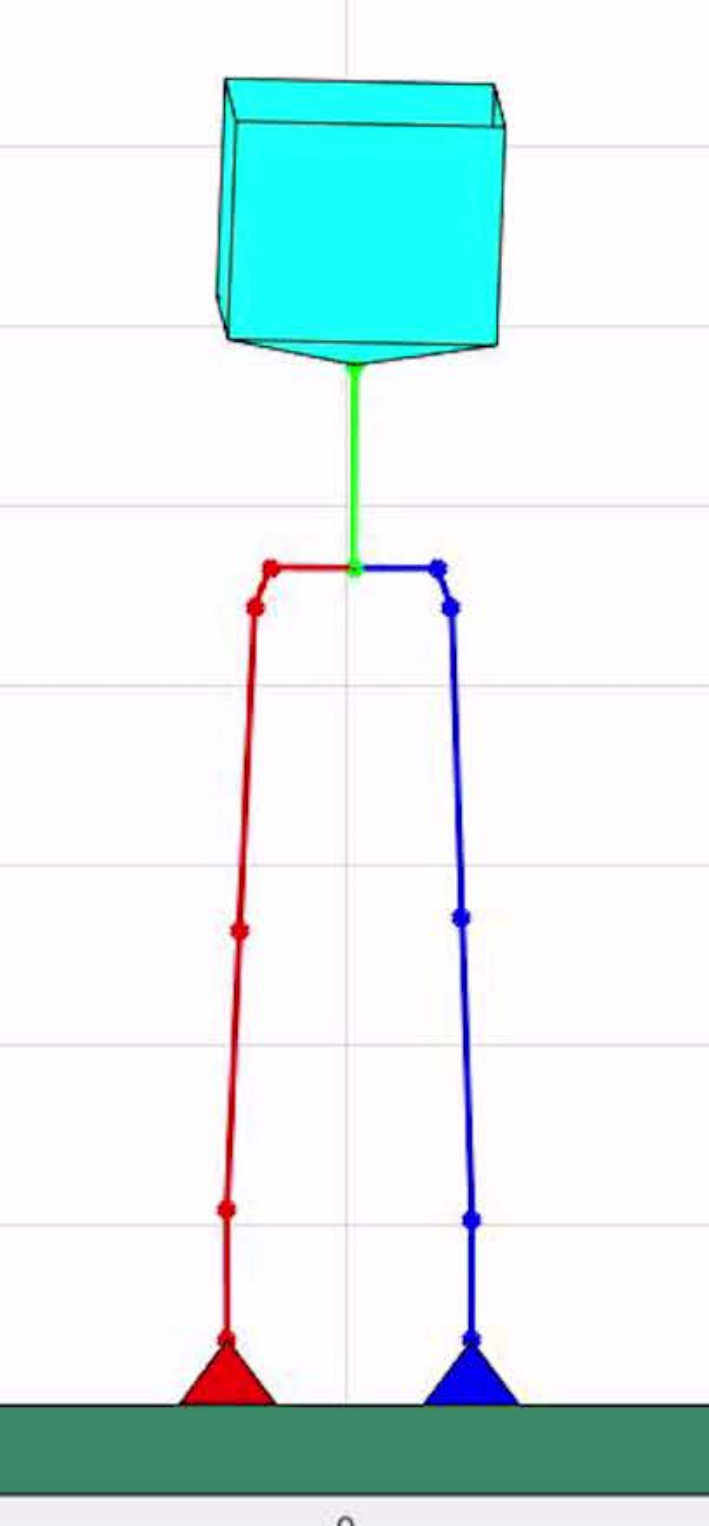}\\
	\includegraphics[width=0.14\textwidth]{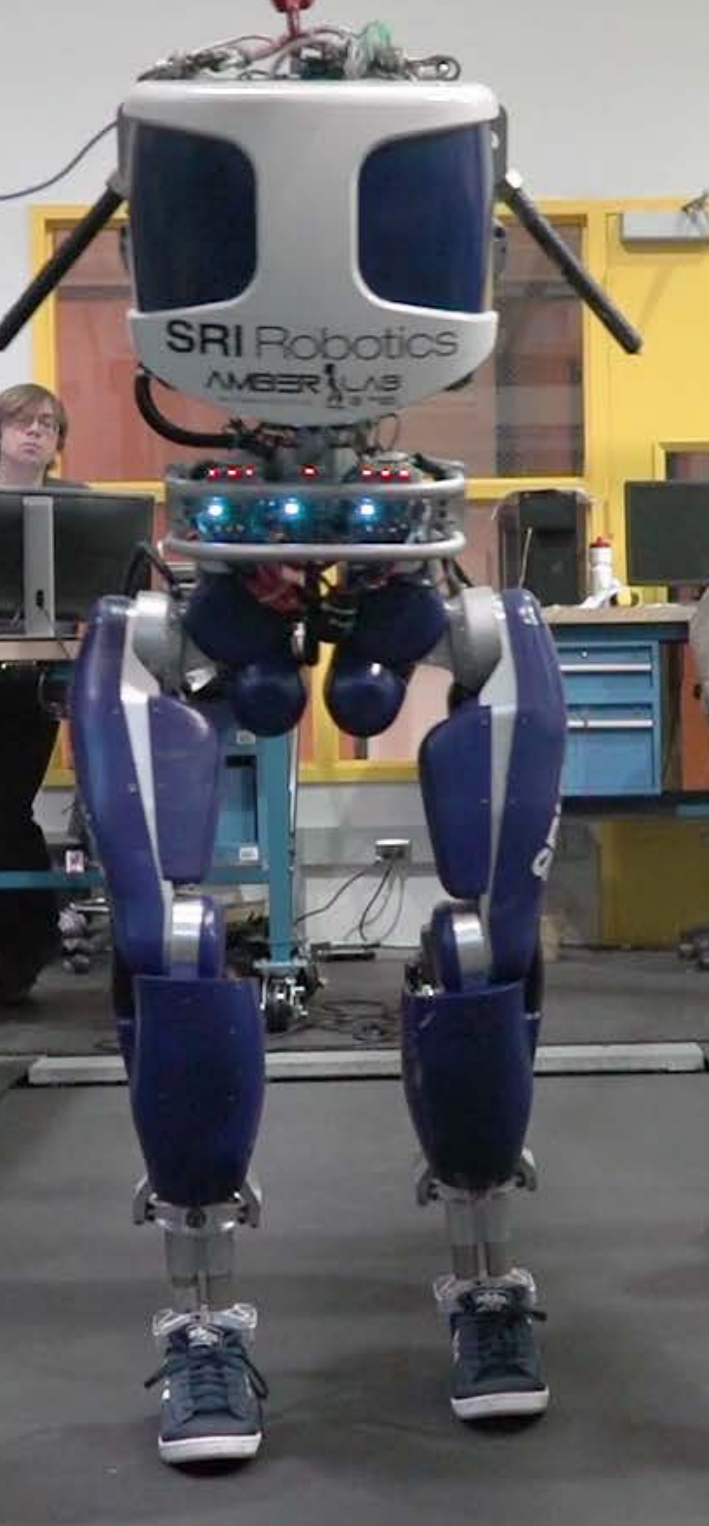}\hspace{-1mm}
	\includegraphics[width=0.14\textwidth]{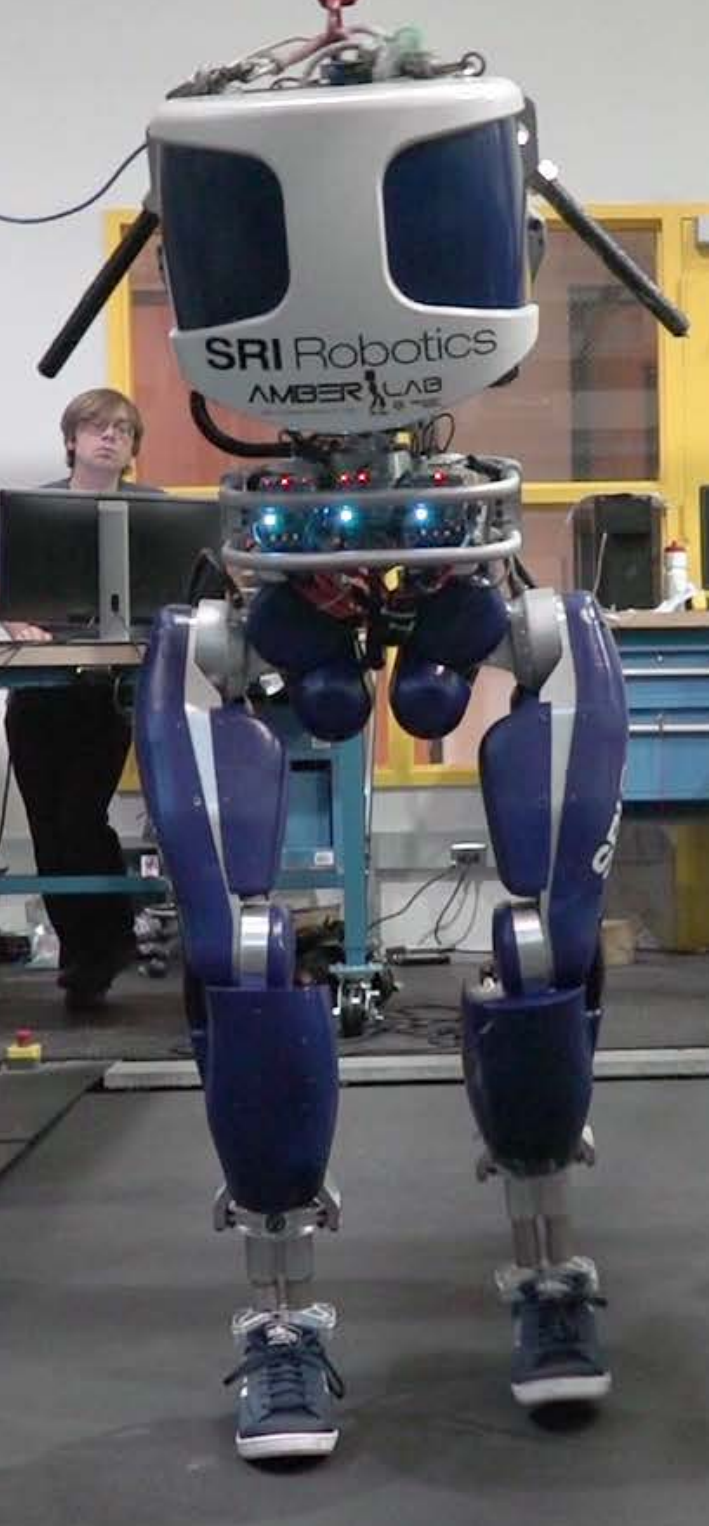}\hspace{-1mm}
	\includegraphics[width=0.14\textwidth]{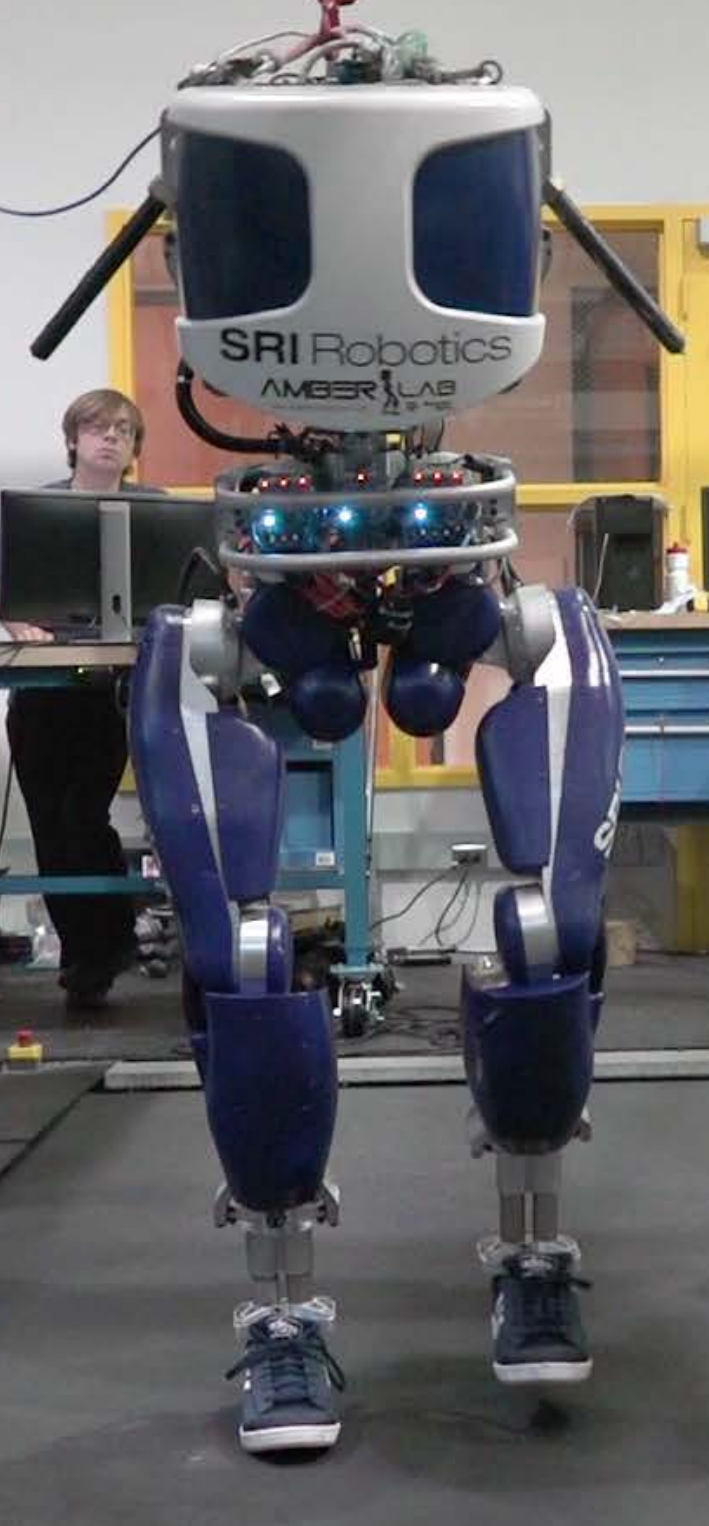}\hspace{-1mm}
	\includegraphics[width=0.14\textwidth]{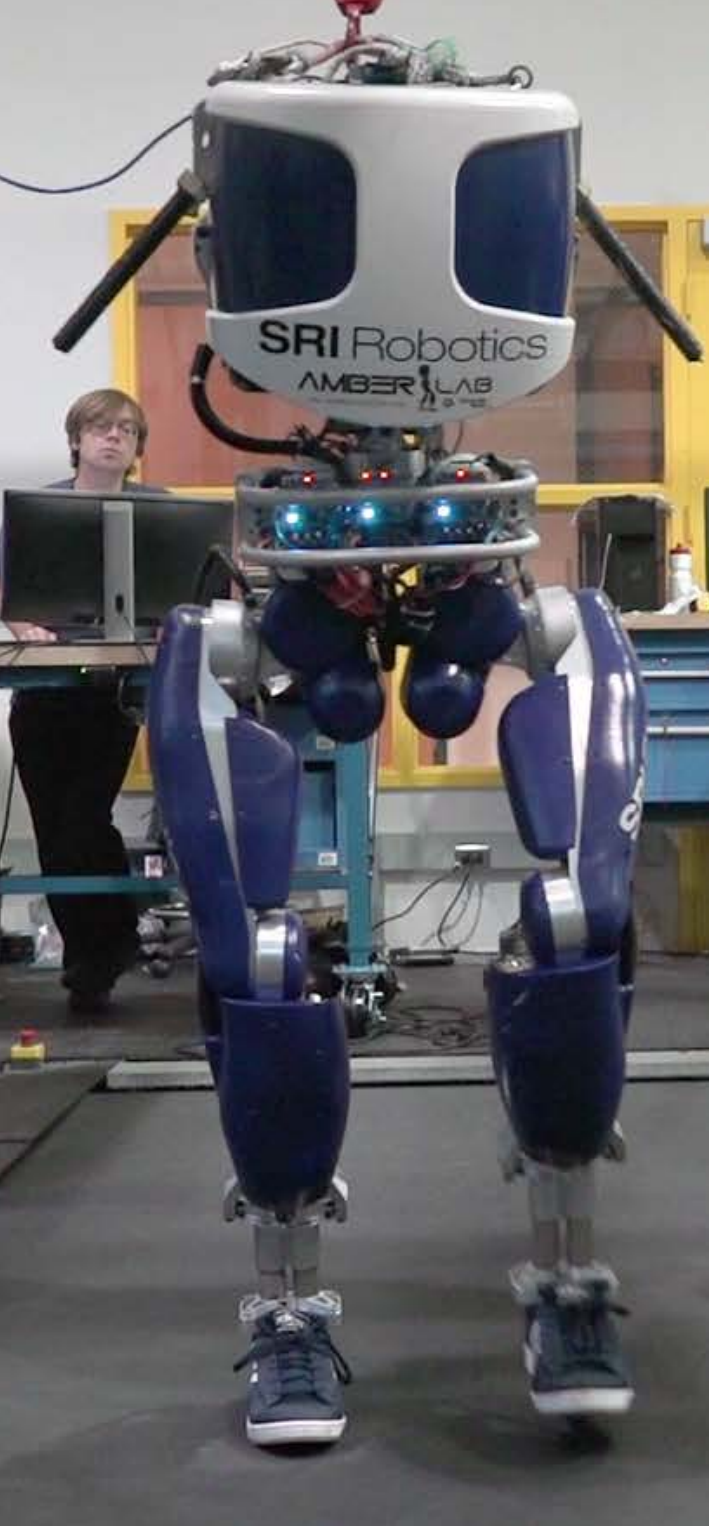}\hspace{-1mm}
	\includegraphics[width=0.14\textwidth]{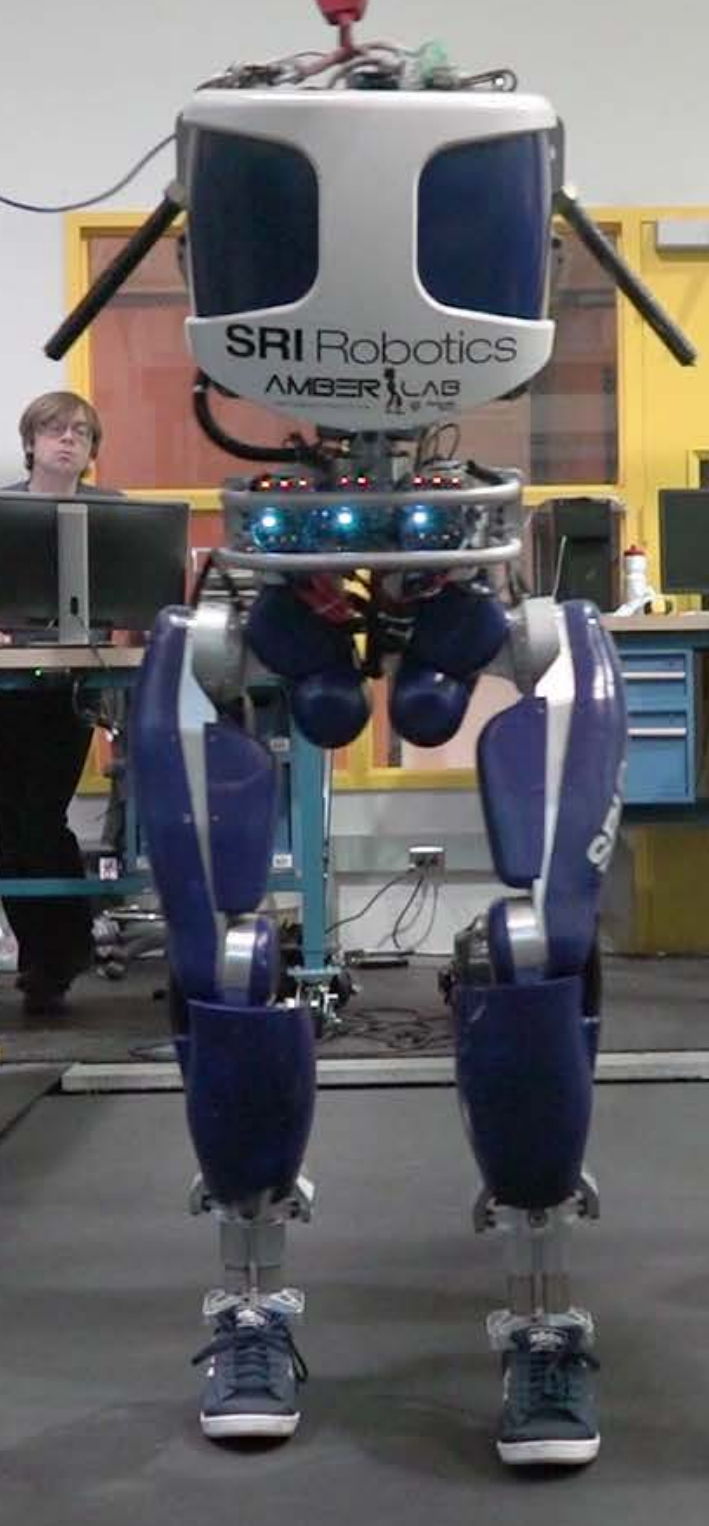}\hspace{-1mm}
	\includegraphics[width=0.14\textwidth]{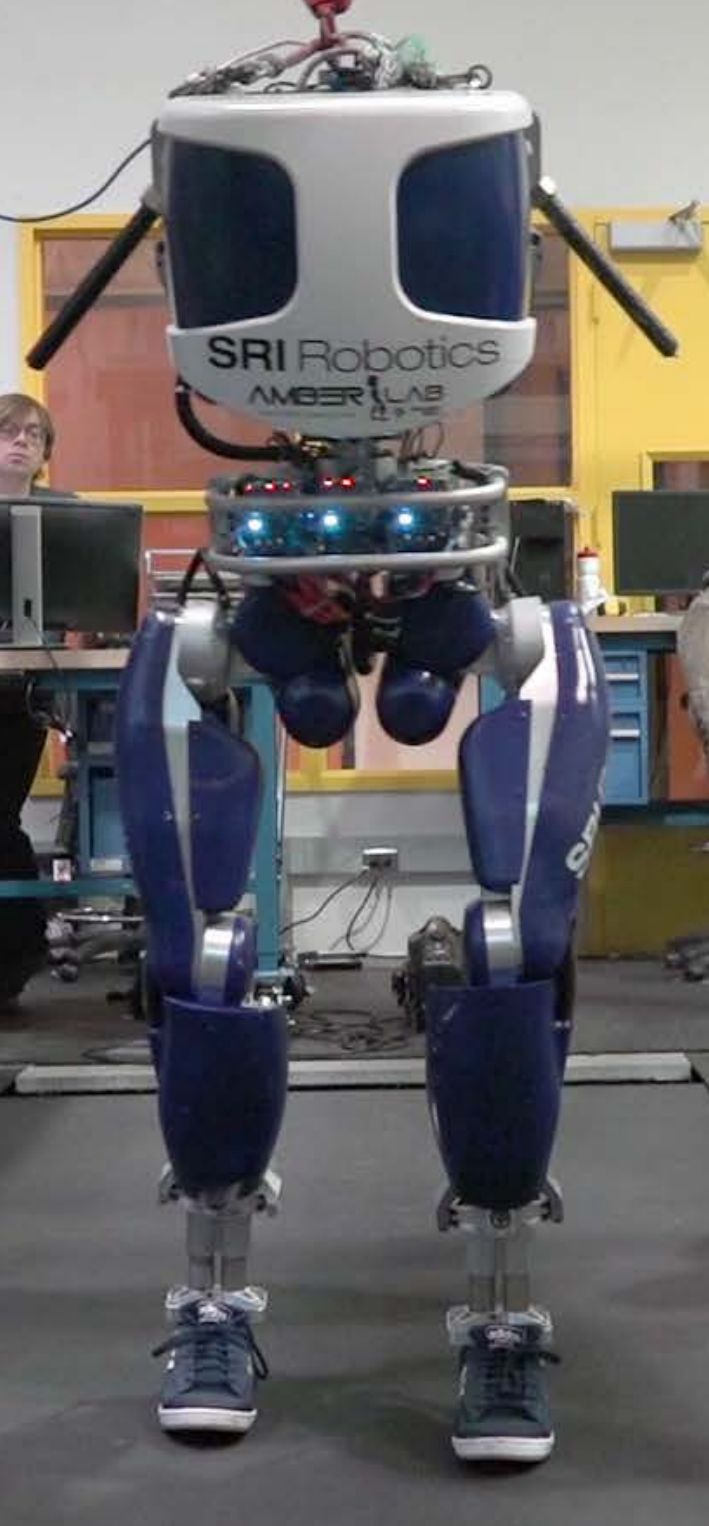}
	\caption{Figure showing the gait tiles for one step for simulation (top) and experiment (bottom).}
	\label{fig:tiles}
\end{figure*}
with $\q^+$ the post-impact configuration of the robot.
Depending upon the domain, some of the outputs are either included or omitted from the output vector to avoid singularity. As observed in \figref{fig:resultssimexp}, the nonstance slope and swing leg roll are not included in the vector $y^a_2$ during the double support ($\dos$) phase. 


\begin{figure}[h!]
	\includegraphics[width=0.98\columnwidth]{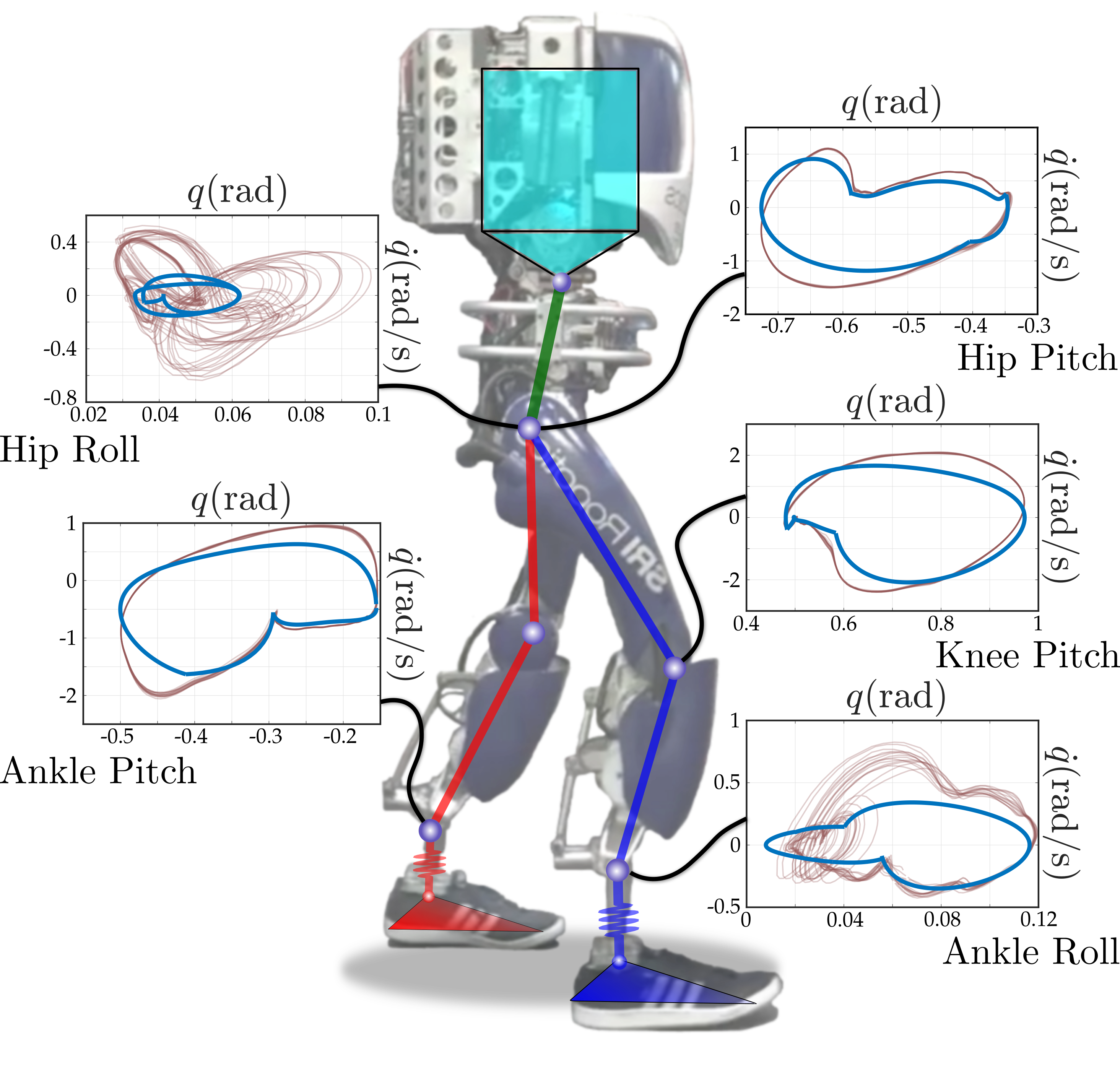} 
	\caption{Periodic orbits of the select joints are shown. Blue waveforms are from the simulation and red waveforms are from the experiment (30 steps).}
	\label{fig:resultsexpPO}
\end{figure}

\newsec{Control law.} 
Given the choice of the outputs, the goal is to obtain a control law that drives these outputs to ``practically'' small values (if not $0$). We choose time based parameterization of the outputs along with PHZD reconstruction to obtain the desired state $(q^t_d,\dq^t_d)$ (see Section \ref{sec:PHZDrecons}). This is utilized to derive the feedback control law \eqref{eq:linearfeedbacklawt} in the robot.


\begin{figure}[ht!]
\centering
	\includegraphics[width=0.7\columnwidth]{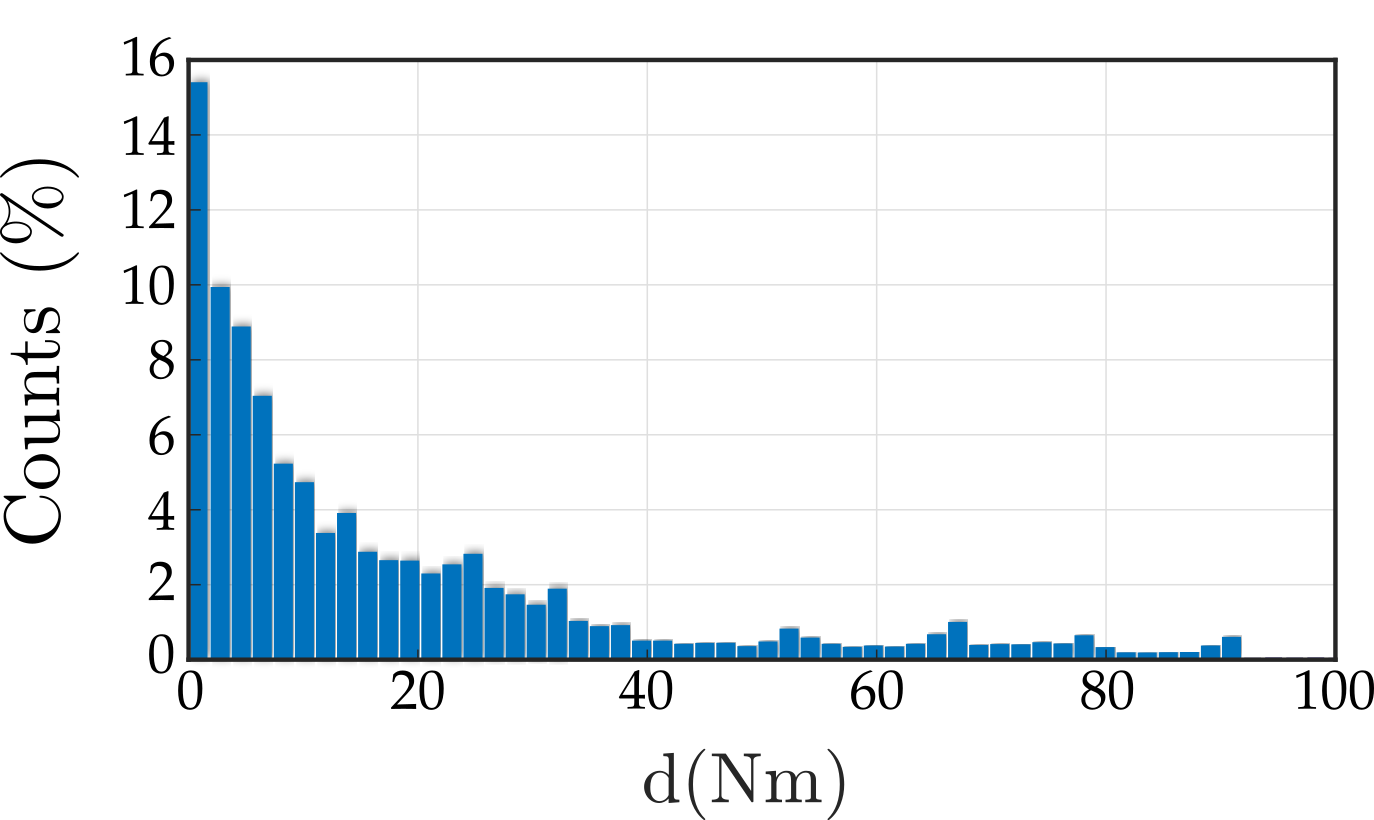} 
	\caption{The histogram of the disturbance input is provided here. High torque disturbances are very rare and the values usually stay within $40$ Nm.}
	\label{fig:resultsdist}
\end{figure}

Because the output combinations chosen for this behavior are nonlinear, the inverse diffeomorphism is reconstructed via an inverse kinematics solver utilizing a Jacobian pseudoinverse. 
The desired states are then communicated to the PD controllers \eqref{eq:linearfeedbacklawt} for individual joint level tracking. The PD gains at the joint level are manually tuned to reach the desired performance. 
The simulation and experimental results are shown in Figs. \ref{fig:resultstau} - \ref{fig:resultsdist} and an experimental video is provided in \cite{videodurus2017isswalking}. The comparison between the time and state based phase variables are shown in \figref{fig:resultstau}. \figref{fig:resultssimexp} shows the comparison between $9$ actual and desired outputs for both simulation (left) and experiment (right). The simulation figures are recorded for one step while the experiment figures are recored for 30 steps of the robot. The remaining 5 outputs are not shown since they do not vary significantly. \figref{fig:resultsexpPO} shows the phase portraits for $5$ outputs of the robot, and \figref{fig:tiles} shows the gait tiles for one step.
\figref{fig:resultsdist} is the histogram of the disturbance input ($u^t_{\rm{PD}} - u_{\rm{IO}}$) observed in simulation. It can be seen that the disturbance values can go as high as $100$ Nm. Note that disturbance input evaluation in the experiment is not possible due to lack of model information. Given this input deviation (disturbance), the maximum deviation of the walking gait from the nominal gait observed in simulation is close to $0.15$ rad and that observed in experiment is $0.5$ rad. 

\newsec{Conclusions.} In this work, it was shown how to obtain robust walking controllers via the notion of input to state stability (ISS) for the humanoid robot DURUS.
With this construction, we obtained the class of input to state stabilizing (ISSing) controllers that yields stable and robust hybrid periodic orbits. It is important to note that ISSing controllers can complement other robust control approaches. In fact, a wide variety of controllers like sliding mode, Lyapunov backstepping, model predictive control themselves yield ISS in a wide variety of robotic systems. The major advantage in ISSing controllers is the extensive analysis involved in minimizing output perturbations for a given uncertainty. Future research will involve utilizing this approach to realize more complex locomotion behaviors like 3D running and dancing.




 \begin{appendices}
  \section{PD Tracking in Continuous Robotic Systems}
\label{app:eisspdcontrol}

We will establish that PD control for continuous robotic systems
yield e-ISS.
For a detailed stability analysis for all kinds of PD based controllers see \cite{wen1988new} (especially see Table $1$ in pages $1382$-$1383$). We have the actual configuration $\q $ and the desired configuration $q_d(t)$, which is a function of time only. For the case where $m < n$ (underactuation), we can pick the passive desired angle to be equal to the corresponding actual angle itself. Therefore, irrespective of the degree of actuation, we can obtain constant diagonal matrices $K_p$, $K_d$ in order to apply the PD control law
\begin{align}
 D(\q) \ddq + C(\q,\dq) \dq + G(\q) =  - K_p (q - q_d) - K_d (\dq - \dq_d), \nonumber
\end{align}
where the holonomic constraints are removed (compared to \eqref{eq:eom-general}), the matrix $H$ is split into Coriolis-centrifugal $C(\q,\dq) \in \R^{n \times n}$ (not to be confused with the constants $C_{\vi}$ used in the main body of the paper) and gravity $G(\q)\in\R^n$ matrices. Note that by Property \ref{prop:1}, the Coriolis-centrifugal matrix is bounded by $|C(\q,\dq)| \leq c_c |\dq|$, and the gravity vector is bounded by $|G(\q)| \leq c_c$.

Denote $e(\q,t): = q - q_d(t)$, and therefore we have
\begin{align}
 D(\q) \ddot{e}  = - C(\q,\dq) \dq - G(\q) - K_p e - K_d \dot e - D(\q) \ddot q_d.
\end{align}
We make the following assumptions about the reference trajectories:
\gap
\begin{assumption}
The desired angles $q_d$, velocities $\dq_d$ and accelerations $\ddq_d$ are all bounded by some $\kappa_q>0$.
\end{assumption}
\gap
Consider a Lyapunov candidate (motivated by strict Lyapunov functions in \cite{4790013})
\begin{align}
   V(e,\dot e,\q)  & = V_0(e,\dot e,\q) +  V_{c} (e,\dot e) \nonumber\\
 V_0 (e,\dot e,\q) & = \frac{1}{2}\begin{bmatrix}
			e \\ \dot e 
		       \end{bmatrix}^T \begin{bmatrix}
                                    K_p & \0 \\
                                    \0 & D(\q)
                                   \end{bmatrix} \begin{bmatrix}
						  e \\ \dot e 
						\end{bmatrix} \nonumber\\
 V_{c} (e,\dot e,\q) & =   \kappa(e) e^T D(\q) \dot e,
\end{align}
where $\kappa$ is given by
\begin{align}
	\kappa(e) = \frac{\kappa_0}{1 + |e|} = \frac{\kappa_0}{1 + \sqrt{e^T e}}. 
\end{align}
A sufficiently small choice of $\kappa_0$ makes $V = V_0 + V_c$ positive definite \cite{210795}. For example,
\begin{align}
	\kappa_0 \leq \frac{(\|K_p\|\|D\|)^{\frac{1}{2}}}{\|D\|}.
\end{align}
Taking the derivative of $\dot V_0$ and collecting the terms yields the following (see \cite{210795})
\begin{align}
\label{eq:v0IOSS}
 \dot V_0 & = - \dot e^T K_d \dot e  - \dot e^T C \dot q_d - \dot e^T G - \dot e^T D \ddot q_d, \nonumber \\
	  & \leq - \dot e^T K_d \dot e  + c_c \kappa_q |\dq| |\dot e| + c_c |\dot e| + \bar c_d \kappa_q |\dot e| \nonumber \\
	  & \leq - \dot e^T K_d \dot e  + c_c \kappa_q |\dot e|^2 + c_c \kappa^2_q  |\dot e| + c_c |\dot e| + \bar c_d \kappa_q |\dot e|
\end{align}
The above equation \eqref{eq:v0IOSS}, in fact, satisfies the conditions for quasi input to state stability (qISS) property \cite{angeli1999input}. By letting equal gains across all joints, we have $k_d:=\|K_d\|$, and thus \eqref{eq:v0IOSS} can be reformulated as
\begin{align}
\dot V_0 \leq - \frac{1}{2}(k_d - c_c \kappa_q) |\dot e  |^2 + \frac{(c_c \kappa^2_q  + c_c + \bar c_d \kappa_q )^2}{k_d - c_c \kappa_q},
\end{align}
where we have used the inequality $ - v^2 + vw \leq - \frac{1}{2}v^2 + w^2$. It can be observed that $V_0$ is decreasing for sufficiently large $|\dot e|$. There are no guarantees on the boundedness of $V_0$, but it can be verified that $|\dot e(t)|$ is bounded after a sufficiently long enough time. 

In a similar fashion, we take the derivative of $V_c$
\begin{align}
\label{eq:vcdot}
 \dot V_c & = \kappa \dot e^T D \dot e + \kappa e^T \dot D \dot e + \kappa e^T D \ddot e  + \dot \kappa e^T D \dot e 
\end{align}
The first term in the addendum can be merged with $V_0$, the next two terms reduce to the following
\begin{align}
	\kappa e^T \dot D \dot e  + \kappa e^T &( - C \dot q - G - K_p e - K_d \dot e - D\ddq_d) \nonumber \\ 
							 \leq &\kappa e^T (\dot D  - C) \dot e - \kappa e^T C \dq_d + \kappa c_c |e| - \kappa e^T K_p e \nonumber \\
							& + \kappa_0 k_d |\dot e| + \kappa_0  \bar c_d \kappa_q \nonumber \\
							\leq & - \kappa k_p |e|^2 + \kappa_0  c_m |\dot e|^2 +\kappa_0  (c_m + c_c) \kappa_q |\dot e| + \kappa_0 c_c  \nonumber \\
							& + \kappa_0 k_d |\dot e| + \kappa_0  \bar c_d \kappa_q,
\end{align}
where we used the following properties
\begin{itemize}
	\item $|\dq| = |\dot e + \dq_d| \leq |\dot e | + \kappa_q$
	\item $\kappa |e| \leq \kappa_0$
	\item $\kappa \leq \kappa_0$
	\item $\|\dot D  - C\| \leq c_m |\dot q|$ for some $c_m >0$
	\item $|\kappa e^T (\dot D  - C) \dot e| \leq \kappa_0  c_m |\dot q| |\dot e| \leq \kappa_0  c_m |\dot e|^2 + \kappa_0  c_m \kappa_q |\dot e|$
	\item Assuming equal gains $k_p:=\|K_p\|$\footnote{Not to be confused with the number of outputs $k_{\vi}$} 
\end{itemize}
In a similar fashion, the fourth term in the addendum \eqref{eq:vcdot} yields the following
\begin{align}
	|\dot \kappa e^T D \dot e | \leq \kappa_0 \bar c_d |\dot e|^2.
\end{align}
Finally
 \begin{align}
  \dot V \leq  &  -\kappa k_p  | e|^2 - (k_d - c_c \kappa_q - \kappa \bar c_d - \kappa_0 c_m) | \dot e|^2  \nonumber \\
   & + (\kappa_0  c_c \kappa^2_q   + c_c  + \bar c_d \kappa_q  + \kappa_0 (c_m + c_c) \kappa_q   +\kappa_0 k_d )|\dot e| \nonumber \\
   & + \kappa_0 (c_c + \bar c_d \kappa_q),
 \end{align}
 which, indeed, can be written in the form given by \eqref{eq:ISSdstricter}, thereby establishing e-ISS. It can be verified that with a small enough $\kappa_0$ and large enough $K_p$,$K_d$, the resulting $\dot V = \dot V_0 + \dot V_c$ is negative definite for large values of $(e,\dot e)$. 
\gap
\begin{remark}
 Note that the above analysis did not include precise characterization of the disturbance input (which will be included in future). Intuitively, the disturbance effects are indirect functions of the model and the reference trajectories of the system (like gravity, Coriolis-centrifugal matrices and the trajectories $\q_d$). A more detailed analysis of PD based tracking and its ISS properties pertaining to walking robots will be included in future.
\end{remark}
\gap

\subsection{PD tracking of state based outputs}
If the desired functions $\q_d$ are state dependent, then we have the following control law
\begin{align}
	u = - K_p (q- \q_d(\q)) - K_d (\dq - \dq_d(q,\dq)),
\end{align}
where we can add and subtract the time based control law to obtain the following
\begin{align}
	u =& - K_p (q- \q_d(t)) - K_d (\dq - \dq_d(t)) \nonumber \\
	&\qquad\underbrace{- K_p( \q_d(t)- \q_d(\q) ) - K_d (\dq_d(t) - \dq_d(q,\dq))}_d ,
\end{align}
where we have the new disturbance input $d$ (called the phase based uncertainty in the context of walking robots \cite{kolathaya2016time}), and as long as this value remains small, the resulting system dynamics is ISS.

 \end{appendices}

\bibliographystyle{plain}
\bibliography{bibdata}

\begin{thebibliography}{10}

\bibitem{videodurus2017isswalking}
{DURUS} walking.
\newblock \url{https://youtu.be/ANS3knFC7uY}.

\bibitem{ames2014human}
Aaron~D Ames.
\newblock Human-inspired control of bipedal walking robots.
\newblock {\em Automatic Control, IEEE Transactions on}, 59(5):1115--1130,
  2014.

\bibitem{ames2013towards}
Aaron~D Ames and Matthew Powell.
\newblock Towards the unification of locomotion and manipulation through
  control lyapunov functions and quadratic programs.
\newblock In {\em Control of Cyber-Physical Systems}, pages 219--240. Springer
  International Publishing, 2013.

\bibitem{TAC:amesCLF}
A.D. Ames, K.~Galloway, K.~Sreenath, and J.W. Grizzle.
\newblock Rapidly exponentially stabilizing control {L}yapunov functions and
  hybrid zero dynamics.
\newblock {\em Automatic Control, IEEE Transactions on}, 59(4):876--891, 4
  2014.

\bibitem{angeli1999input}
David Angeli.
\newblock Input-to-state stability of pd-controlled robotic systems.
\newblock {\em Automatica}, 35(7):1285 -- 1290, 1999.

\bibitem{cai2005results}
Chaohong Cai and A.~R. Teel.
\newblock Results on input-to-state stability for hybrid systems.
\newblock In {\em Proceedings of the 44th IEEE Conference on Decision and
  Control}, pages 5403--5408, 12 2005.

\bibitem{6094435}
J.~Englsberger, C.~Ott, M.~A. Roa, A.~Albu-Schäffer, and G.~Hirzinger.
\newblock Bipedal walking control based on capture point dynamics.
\newblock In {\em 2011 IEEE/RSJ International Conference on Intelligent Robots
  and Systems}, pages 4420--4427, Sept 2011.

\bibitem{ROB:ROB2}
Fathi Ghorbel, B.~Srinivasan, and Mark~W. Spong.
\newblock On the uniform boundedness of the inertia matrix of serial robot
  manipulators.
\newblock {\em Journal of Robotic Systems}, 15(1):17--28, 1998.

\bibitem{umich_mabel}
J.~W. Grizzle, J.~Hurst, B.~Morris, H.~Park, and K.~Sreenath.
\newblock {MABEL}, a new robotic bipedal walker and runner.
\newblock In {\em American Control Conference}, pages 2030--2036, St. Louis,
  MO, USA, 2009.

\bibitem{Hereid_etal_2016}
Ayonga Hereid, Eric~A. Cousineau, Christian~M. Hubicki, and Aaron~D. Ames.
\newblock 3d dynamic walking with underactuated humanoid robots: A direct
  collocation framework for optimizing hybrid zero dynamics.
\newblock In {\em 2016 IEEE International Conference on Robotics and Automation
  (ICRA)}, pages 1447--1454, 5 2016.

\bibitem{hespanha2005input}
J.~P. Hespanha, D.~Liberzon, and A.~R. Teel.
\newblock On input-to-state stability of impulsive systems.
\newblock In {\em Proceedings of the 44th IEEE Conference on Decision and
  Control}, pages 3992--3997, 12 2005.

\bibitem{Hubicki2016}
Christian Hubicki, Jesse Grimes, Mikhail Jones, Daniel Renjewski, Alexander
  Spr{\"o}witz, Andy Abate, and Jonathan Hurst.
\newblock Atrias: Design and validation of a tether-free {3D}-capable
  spring-mass bipedal robot.
\newblock {\em The International Journal of Robotics Research}, 2016.

\bibitem{5651082}
S.~Kajita, M.~Morisawa, K.~Miura, S.~Nakaoka, K.~Harada, K.~Kaneko,
  F.~Kanehiro, and K.~Yokoi.
\newblock Biped walking stabilization based on linear inverted pendulum
  tracking.
\newblock In {\em 2010 IEEE/RSJ International Conference on Intelligent Robots
  and Systems}, pages 4489--4496, Oct 2010.

\bibitem{1643376}
J.~Kasac, B.~Novakovic, D.~Majetic, and D.~Brezak.
\newblock Global positioning of robot manipulators with mixed revolute and
  prismatic joints.
\newblock {\em IEEE Transactions on Automatic Control}, 51(6):1035--1040, June
  2006.

\bibitem{7844077}
J.~H. Kim, S.~m.~Hur, and Y.~Oh.
\newblock Maximum tracking errors in pd-controlled robotic manipulators.
\newblock In {\em 2016 IEEE/SICE International Symposium on System Integration
  (SII)}, pages 676--681, Dec 2016.

\bibitem{4790013}
D.~E. Koditschek.
\newblock Strict global lyapunov functions for mechanical systems.
\newblock In {\em 1988 American Control Conference}, pages 1770--1775, June
  1988.

\bibitem{koditschek-the_robotics_review-1989}
Daniel Koditschek.
\newblock Robot planning and control via potential functions.
\newblock pages 349--367, 1989.

\bibitem{kolathaya2016time}
S.~Kolathaya, A.~Hereid, and A.~D. Ames.
\newblock Time dependent control {L}yapunov functions and hybrid zero dynamics
  for stable robotic locomotion.
\newblock In {\em 2016 American Control Conference (ACC)}, pages 3916--3921, 7
  2016.

\bibitem{kolathaya2016parameter}
Shishir Kolathaya and Aaron~D. Ames.
\newblock Parameter to state stability of control {L}yapunov functions for
  hybrid system models of robots.
\newblock {\em Nonlinear Analysis: Hybrid Systems}, 25:174 -- 191, 2017.

\bibitem{hscc17running}
Wen-Loong Ma, Shishir Kolathaya, Eric~R. Ambrose, Christian~M. Hubicki, and
  Aaron~D. Ames.
\newblock Bipedal robotic running with durus-2d: Bridging the gap between
  theory and experiment.
\newblock In {\em Proceedings of the 20th International Conference on Hybrid
  Systems: Computation and Control}, HSCC '17, pages 265--274, New York, NY,
  USA, 2017. ACM.

\bibitem{MOGR05}
{B.} Morris and {J.W.} Grizzle.
\newblock A restricted {P}oincar\'e map for determining exponentially stable
  periodic orbits in systems with impulse effects: Application to bipedal
  robots.
\newblock In {\em IEEE Conf. on Decision and Control}, Seville, Spain, 2005.

\bibitem{morris2009hybrid}
Benjamin Morris and Jessy~W Grizzle.
\newblock Hybrid invariant manifolds in systems with impulse effects with
  application to periodic locomotion in bipedal robots.
\newblock {\em Automatic Control, IEEE Transactions on}, 54(8):1751--1764,
  2009.

\bibitem{RSS2015_RobustCLF}
Quan Nguyen and Koushil Sreenath.
\newblock Optimal robust control for bipedal robots through control {L}yapunov
  function based quadratic programs.
\newblock In {\em Robotics: Science and Systems (RSS)}, Rome, Italy, 7 2015.

\bibitem{doi:10.5772/61537}
Jorge Orrante-Sakanassi, Victor M.~Hernandez Guzman, and Victor Santibanez.
\newblock New tuning conditions for semiglobal exponential stability of the
  classical pid regulator for rigid robots.
\newblock {\em International Journal of Advanced Robotic Systems}, 12(10):143,
  2015.

\bibitem{371478}
Z.~Qu, D.~M. Dawson, J.~F. Dorsey, and S.~Y. Lim.
\newblock A new class of robust control laws for tracking of robots.
\newblock In {\em [1992] Proceedings of the 31st IEEE Conference on Decision
  and Control}, pages 1408--1409 vol.2, 1992.

\bibitem{ACMRAIBERT}
M.~H. Raibert.
\newblock Legged robots.
\newblock {\em Communications of the ACM}, 29(6):499--514, 1986.

\bibitem{reheralgorithmic}
Jacob~P Reher, Ayonga Hereid, Shishir Kolathaya, Christian~M Hubicki, and
  Aaron~D Ames.
\newblock {\em Algorithmic Foundations of Realizing Multi-Contact Locomotion on
  the Humanoid Robot {DURUS}}.
\newblock Springer Berlin Heidelberg, Berlin, Heidelberg, 2017.

\bibitem{sontag1989smooth}
E.~D. Sontag.
\newblock Smooth stabilization implies coprime factorization.
\newblock {\em IEEE Transactions on Automatic Control}, 34(4):435--443, 4 1989.

\bibitem{Sontag89furtherfacts}
Eduardo~D. Sontag.
\newblock Further facts about input to state stabilization.
\newblock {\em IEEE Trans. Automat. Control}, 35:473--476, 1989.

\bibitem{sontag1991input}
Eduardo~D Sontag.
\newblock Input/output and state-space stability.
\newblock In {\em New Trends in Systems Theory}, pages 684--691. Springer,
  1991.

\bibitem{sontag2008input}
Eduardo~D. Sontag.
\newblock {\em Input to State Stability: Basic Concepts and Results}, pages
  163--220.
\newblock Springer Berlin Heidelberg, 2008.

\bibitem{sontag1995characterizations}
Eduardo~D Sontag and Yuan Wang.
\newblock On characterizations of the input-to-state stability property.
\newblock {\em Systems \& Control Letters}, 24(5):351--359, 1995.

\bibitem{veer2017poincare}
Sushant Veer, Rakesh, and Ioannis Poulakakis.
\newblock Poincare analysis of hybrid periodic orbits of systems with impulse
  effects under external inputs.
\newblock {\em arXiv preprint arXiv:1712.03291}, 2017.

\bibitem{VB05}
M.~Vukobratovi{\'c} and B.~Borovac.
\newblock Zero-moment point---thirty-five years of its life.
\newblock {\em Intl. J. of Humanoid Robotics}, 1(1):157--173, 2005.

\bibitem{wen1988new}
John~T Wen and David~S Bayard.
\newblock New class of control laws for robotic manipulators part 1.
  non--adaptive case.
\newblock {\em International Journal of Control}, 47(5):1361--1385, 1988.

\bibitem{WGCCM07}
E.R. Westervelt, J.W. Grizzle, C.~Chevallereau, J.H. Choi, and B.~Morris.
\newblock {\em Feedback Control of Dynamic Bipedal Robot Locomotion}.
\newblock Automation and Control Engineering. CRC Press, 2007.

\bibitem{210795}
L.~L. Whitcomb, A.~A. Rizzi, and D.~E. Koditschek.
\newblock Comparative experiments with a new adaptive controller for robot
  arms.
\newblock {\em IEEE Transactions on Robotics and Automation}, 9(1):59--70, Feb
  1993.

\bibitem{Yadukumar2013}
Shishir~Nadubettu Yadukumar, Murali Pasupuleti, and Aaron~D. Ames.
\newblock {\em From Formal Methods to Algorithmic Implementation of Human
  Inspired Control on Bipedal Robots}, pages 511--526.
\newblock Springer Berlin Heidelberg, Berlin, Heidelberg, 2013.

\end{thebibliography}

\begin{wrapfigure}{l}{3.2cm}
	\vspace{-4mm}
	\includegraphics[width=0.175\textwidth]{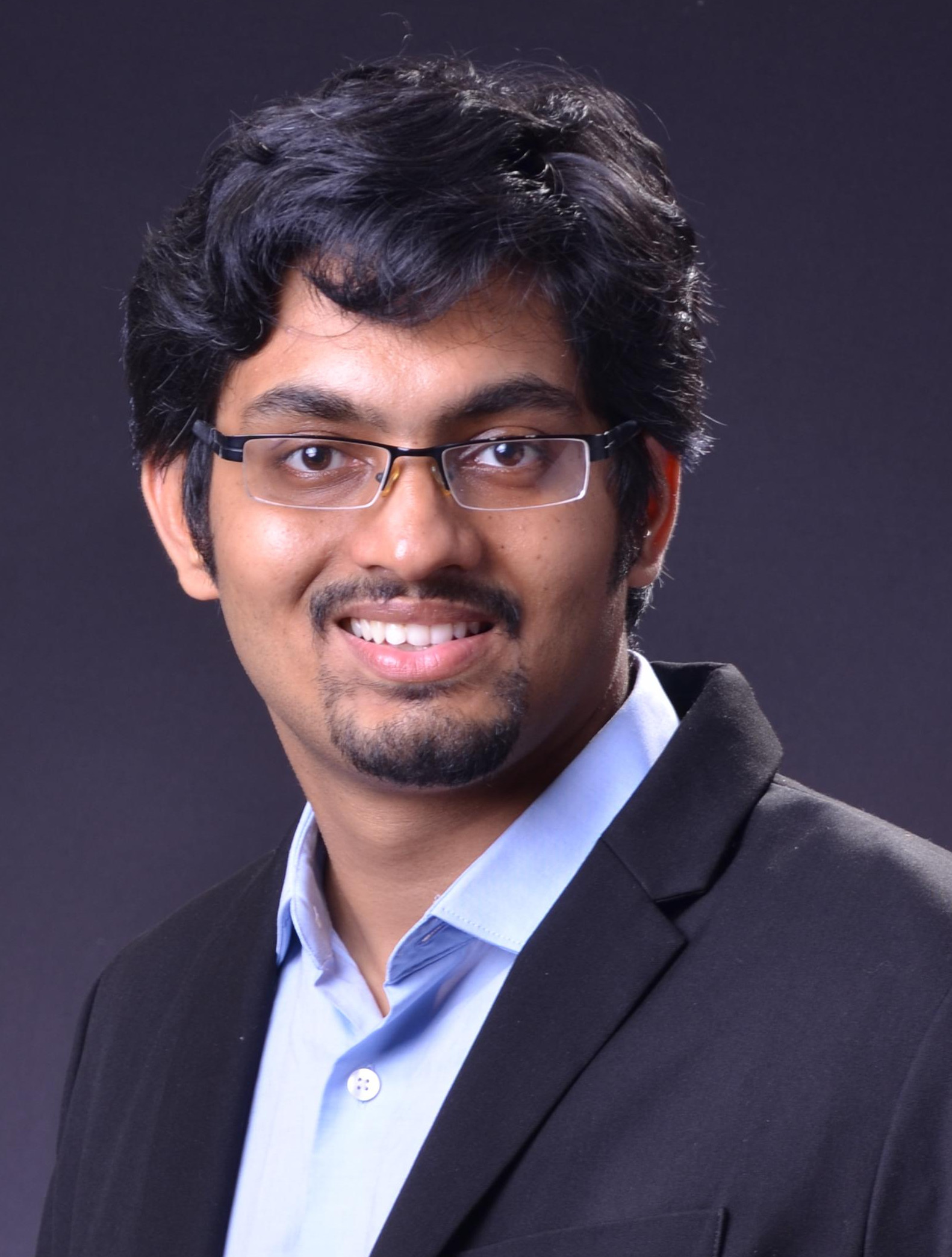}
\end{wrapfigure}
\vspace{4mm}
\noindent \textbf{Shishir Kolathaya} is the Postdoctoral scholar of Mechanical and Civil Engineering at Caltech. He received his Ph.D. degree in Mechanical Engineering (2016) from the Georgia Institute of Technology, M.S. degree in Electrical Engineering (2012) from Texas A\&M University. He received his B.Tech. degree in Electrical Engineering (2008) from the National Institute of Technology, Surathkal. Shishir is interested in nonlinear stability and control of hybrid systems, especially in the domain of walking robots.

\begin{wrapfigure}{l}{3.2cm}
	\vspace{-4mm}
	\includegraphics[width=0.17\textwidth]{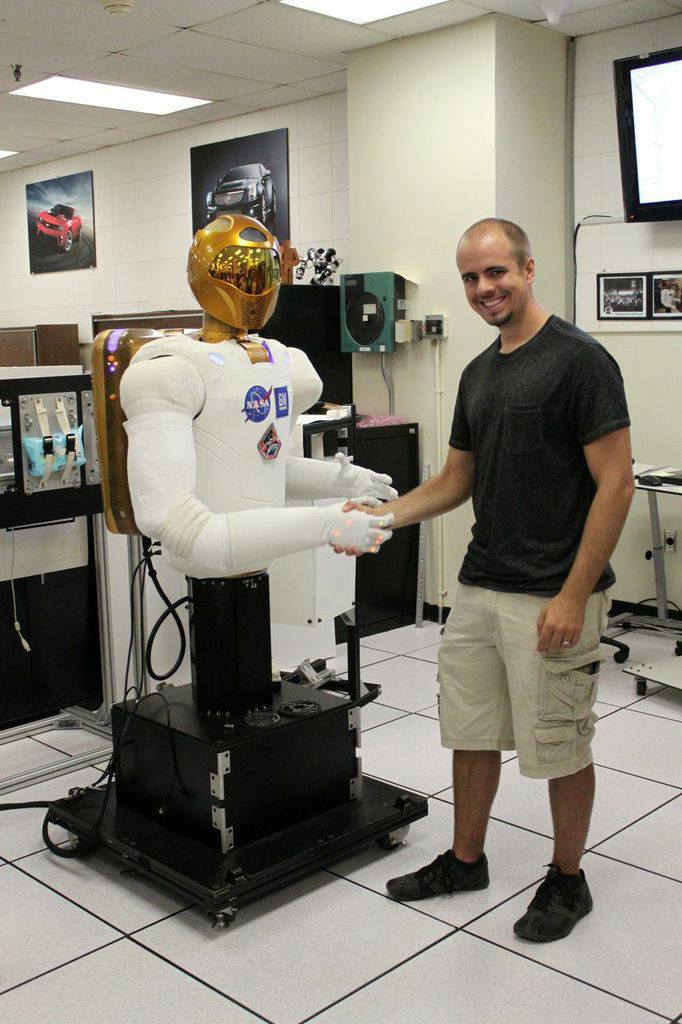}
\end{wrapfigure}
\vspace{4mm}
\noindent \textbf{Jacob Reher} is the graduate student of Mechanical and Civil Engineering at Caltech. He received his B.S. degree in Mechanical Engineering (2013) from the University of Nebraska, Nebraska. Jacob is interested in stability and control of walking robots. His special focus is on realizing human-like behaviors like heel-lift and toe-strikes with the end result being remarkably low cost of transport.

\begin{wrapfigure}{l}{3.2cm}
	\vspace{-4mm}
	\includegraphics[width=0.175\textwidth]{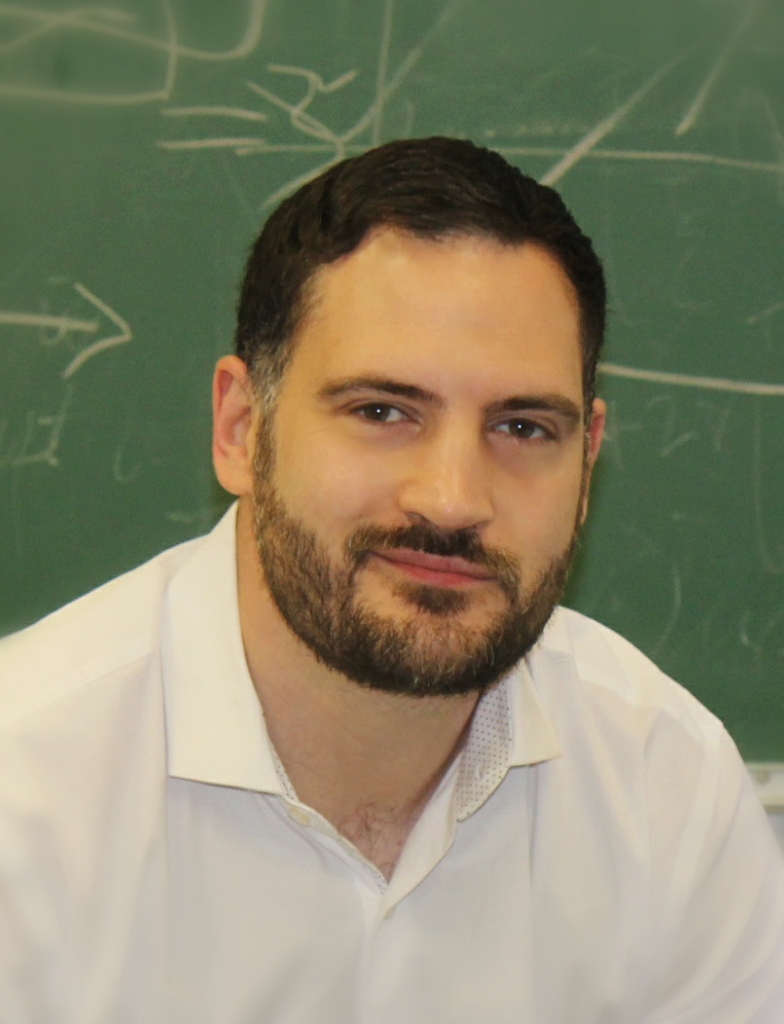}
\end{wrapfigure}
\vspace{4mm}
\noindent \textbf{Aaron D. Ames} is the Bren Professor of Mechanical and Civil Engineering and Control and Dynamical Systems at Caltech.   Prior to joining Caltech in 2017, he was an Associate Professor at Georgia Tech in the Woodruff School of Mechanical Engineering and the School of Electrical \& Computer Engineering.  He received a B.S. in Mechanical Engineering and a B.A. in Mathematics from the University of St. Thomas in 2001, and he received a M.A. in Mathematics and a Ph.D. in Electrical Engineering and Computer Sciences from UC Berkeley in 2006.  He served as a Postdoctoral Scholar in Control and Dynamical Systems at Caltech from 2006 to 2008, and began his faculty career at Texas A\&M University in 2008.  At UC Berkeley, he was the recipient of the 2005 Leon O. Chua Award for achievement in nonlinear science and the 2006 Bernard Friedman Memorial Prize in Applied Mathematics, and he received the NSF CAREER award in 2010 and the 2015 Donald P. Eckman Award. 

\end{document}